\documentclass{article}   	
\usepackage{geometry}               
\usepackage{graphicx,color}						
\usepackage{amssymb,amsmath,amsthm,amsfonts}
 
\usepackage{bm}
\usepackage[english]{babel}
\usepackage[latin1]{inputenc}
\usepackage[colorlinks,pdfencoding=auto]{hyperref}

\newtheorem{proposition}{Proposition}[section]
\newtheorem{theorem}[proposition]{Theorem}
\newtheorem{corollary}[proposition]{Corollary}
\newtheorem{lemma}[proposition]{Lemma}

\theoremstyle{definition}
\newtheorem{definition}[proposition]{Definition}
\newtheorem{remark}[proposition]{Remark}
\numberwithin{equation}{section}

\newcommand{\eps}{\varepsilon}

\newcommand{\N}{{\mathbb{N}}}
\newcommand{\Z}{{\mathbb{Z}}}

\newcommand{\R}{{\mathbb{R}}}

\newcommand{\sphere}{{\mathbb{S}}}
\newcommand{\from}{\colon}
\newcommand{\loc}{{\mathrm{loc}}}
\newcommand{\setp}{{\mathfrak{A}}}

\def\Ecal{{\mathcal{E}}}

\def\Mcal{{\mathcal{M}}}

\def\Ucal{{\mathcal{U}}}
\def\Om{\Omega}
\def\om{\Omega}

\DeclareMathOperator{\spann}{span}

\DeclareMathOperator{\proj}{proj}

\title{Normalized bound states for the nonlinear Schr\"odinger equation
in bounded domains}
\author{Dario Pierotti and Gianmaria Verzini}

\begin{document}
\maketitle

\begin{abstract}
Given $\rho>0$, we study the elliptic problem
\[
\text{find } (U,\lambda)\in H^1_0(\Omega)\times \R
\text{ such that }
\begin{cases}
-\Delta U+\lambda U=|U|^{p-1}U\\
\int_{\Omega} U^2\, dx=\rho,
\end{cases}
\]
where $\Omega\subset\R^N$ is a bounded domain and $p>1$ is Sobolev-subcritical, searching
for conditions (about $\rho$, $N$ and $p$) for the existence of solutions. By the
Gagliardo-Nirenberg inequality it follows that, when $p$ is
$L^2$-subcritical, i.e. $1<p\leq1+4/N$, the problem admits solution for every $\rho>0$. In
the $L^2$-critical and supercritical case, i.e. when $1+4/N \leq p < 2^*-1$, we
show that, for any $k\in\N$, the problem admits solutions having Morse index bounded above by
$k$ only if $\rho$ is sufficiently small. Next we provide existence results for certain
ranges of $\rho$, which can be estimated in terms of the Dirichlet eigenvalues of $-\Delta$
in $H^1_0(\Omega)$, extending to general domains and to changing sign solutions some results
obtained in  \cite{MR3318740} for positive solutions in the ball.
\end{abstract}

\noindent
{\footnotesize \textbf{AMS-Subject Classification}}. {\footnotesize 35Q55, 35J20, 35C08}
\section{Introduction}

Given $\rho>0$, we consider the problem
\begin{equation}\label{eq:main_prob_U}
\begin{cases}
-\Delta U + \lambda U = |U|^{p-1}U  & \text{in }\Omega,\smallskip\\
\int_\Omega U^2\,dx = \rho, \quad U=0  & \text{on }\partial\Omega,
\end{cases}
\end{equation}
where $\Omega\subset\R^N$ is a Lipschitz, bounded domain, $1<p<2^*-1$, $\rho>0$ is a fixed
parameter, and both $U\in H^1_0(\Omega)$ and $\lambda\in\R$ are unknown. More precisely,
we investigate conditions on $p$ and $\rho$ (and also $\Omega$) for the solvability of
the problem.

The main interest in \eqref{eq:main_prob_U} relies on the investigation of standing wave
solutions for the nonlinear Schr\"odinger equation
\[
i\frac{\partial \Phi}{\partial t}+\Delta \Phi+ |\Phi|^{p-1}\Phi=0,\qquad
(t,x)\in \R\times \Omega
\]
with Dirichlet boundary conditions on $\partial\Omega$. This equation appears in several
different physical models, both in the case $\Omega=\R^N$ \cite{MR2002047}, and on bounded
domains \cite{MR1837207}. In particular, the latter case appears in nonlinear optics and in
the theory of Bose-Einstein condensation, also as a limiting case of the equation on $\R^N$
with confining potential. When searching for solutions having the
wave function $\Phi$ factorized as $\Phi(x,t)=e^{i\lambda t} U(x)$, one obtains
that the real valued function $U$ must solve
\begin{equation}\label{eq:NLS}
-\Delta U + \lambda U = |U|^{p-1}U ,\qquad U\in H^1_0(\Omega),
\end{equation}
and two points of view are available. The first possibility is to assign the chemical
potential $\lambda\in\R$, and search for solutions of \eqref{eq:NLS} as critical points of the
related action functional. The literature concerning this approach is huge and we do not even
make an attempt to summarize it here. On the contrary, we focus on the second possibility,
which consists in considering $\lambda$ as part of the unknown and prescribing the mass  (or
charge) $\|U\|_{L^2(\Omega)}^2$ as a natural additional condition. Up to our knowledge,
the only previous paper dealing with this case, in bounded domains, is \cite{MR3318740},
which we describe below. The problem of searching for normalized solutions in $\R^N$,
with non-homogeneous nonlinearities, is more investigated \cite{MR3009665,MR1430506},
even though the methods used there can not be easily extended to bounded domains, where
dilations are not allowed. Very recently, also the case of partial confinement has been
considered \cite{BeBoJeVi_2016}.

Solutions of \eqref{eq:main_prob_U} can be identified with critical points of the
associated energy functional
\[
\mathcal{E}(U) = \frac12\int_\Omega|\nabla U|^2\,dx - \frac{1}{p+1}  \int_\Omega|U|^{p+1}\,dx
\]
restricted to the mass constraint
\[
\Mcal_\rho=\{U\in H_0^1(\Om) : \|U\|_{L^2(\Om)}=\rho\},
\]
with $\lambda$ playing the role of a Lagrange multiplier.

A cricial role in the discussion of the above problem is played by the Gagliardo-Nirenberg
inequality: for any $\Omega$ and for any $v\in H^1_0(\Omega)$,
\begin{equation}
\label{sobest}
\|v\|^{p+1}_{L^{p+1}(\Omega)} \leq C_{N,p} \| \nabla v \|_{L^2(\Omega)}^{N(p-1)/2}
\| v \|_{L^2(\Omega)} ^{(p+1)-N(p-1)/2},
\end{equation}
the equality holding only when $\Omega=\R^N$ and $v=Z_{N,p}$, the positive solution of $-\Delta Z + Z = Z^{p}$ (which is unique up to
translations \cite{MR969899}).
Accordingly, the exponent $p$ can be classified in relation with the so called
\emph{$L^2$-critical exponent} $1+4/N$ (throughout all the paper, $p$ will be always
Sobolev-subcritical and its criticality will be understood in the $L^2$ sense).
Indeed we have that $\Ecal$ is bounded below and coercive on $\Mcal_\rho$ if and only if
either $p$ is subcritical, or it is critical and $\rho$ is sufficiently small.

The recent paper \cite{MR3318740} deals with problem \eqref{eq:main_prob_U} in the case of the spherical domain $\Omega = B_1$, when searching for positive solutions $U$.
In particular, it is shown that the solvability of \eqref{eq:main_prob_U} is strongly influenced by the exponent $p$, indeed:
\begin{itemize}
 \item in the subcritical case $1<p<1+4/N$, \eqref{eq:main_prob_U} admits a unique positive
 solution for every $\rho>0$;
 \item if $p=1+4/N$ then \eqref{eq:main_prob_U} admits a unique
 positive solution for
 \[
 0<\rho<\rho^*=\left(\frac{p+1}{2C_{N,p}}\right)^{N/2}=\|Z_{N,p}\|^2_{L^2(\R^N)},
 \]
 and no positive solutions for $\rho\geq\rho^*$;
 \item finally, in the supercritical regime $1+4/N<p<2^*-1$, \eqref{eq:main_prob_U} admits positive
 solutions if and only if $0<\rho\leq\rho^*$ (the threshold $\rho^*$ depending on $p$), and such
 solutions are at least two for $\rho<\rho^*$.
\end{itemize}

In this paper we carry on such analysis, dealing with a general domain $\Omega$ and with solutions which
are not necessarily positive.
More precisely, let us recall that for any $U$ solving \eqref{eq:main_prob_U} for some $\lambda$, it is well-defined the Morse index
\[
m(U) = \max\left\{k : \begin{array}{l}
\exists V\subset H^1_0(\Omega),\,\dim(V)= k:\forall v\in V\setminus\{0\}\smallskip\\
\displaystyle\int_\Omega |\nabla v|^2 + \lambda v^2 - p|U|^{p-1}v^2\,dx<0
\end{array}
\right\}\in\N.
\]
Then, if $\Omega=B_1$, it is well known that a solution $U$ of \eqref{eq:main_prob_U} is
positive if and only if $m(U)=1$. Under this perspective, the results in \cite{MR3318740}
can be read in terms of Morse index one--solutions, rather than positive ones: introducing the sets of admissible masses
\[
\setp_k =\setp_k(p,\Omega) := \left\{\rho>0 : \begin{array}{l} \eqref{eq:main_prob_U}
\text{ admits a solution $U$ (for some $\lambda$)}\\ \text{having Morse index }m(U)\leq k
\end{array} \right\},
\]
then \cite{MR3318740} implies that $\setp_1(p,B_1)$ is a bounded interval if
and only if $p$ is critical or supercritical, while $\setp_1(p,B_1)=\R^+$ in the subcritical case.
On the contrary, when considering general domains and higher Morse index, the situation may become
much more complicated. We collect some examples in the following remark.
\begin{remark}\label{rem:specialdomains}
In the case of a symmetric domain, one can use any solution as a building block to construct other solutions with
a more complex behavior, obtaining the so-called necklace solitary waves. Such kind of solutions are constructed
in \cite{MR3426917}, even though in such paper the focus is on stability, rather than on normalization conditions.
For instance, by scaling argument, any Dirichlet solution of $-\Delta U + \lambda U = |U|^{p-1}U$ in a rectangle $R=\prod_{i=1}^N(a_i,b_i)$ can
be scaled to a solution of $-\Delta U + k^2\lambda U = |U|^{p-1}U$ in $R/k$, $k\in\N_+$, and then $k^N$ copies of it can be
juxtaposed, with alternating sign. In this way one obtains a new solution on $R$ having $k^{4/(p-1)}$ times the mass of the starting one,
and eventually solutions in $R$ with arbitrarily high mass (but with higher Morse index) can be constructed even in the critical
and supercritical case. An analogous construction can be performed in the disk, using solutions in circular sectors as building blocks,
even though in this case explicit bounds on the mass obtained are more delicate.
Also, instead of symmetric domains, singular perturbed ones can be considered, such as dumbbell domains
\cite{MR949628}: for instance, using \cite[Theorem 3.5]{MR2997381}, one can show that for any $k$, there exists a domain $\Omega$,
which is close in a suitable sense to the disjoint union of $k$ domains, such that \eqref{eq:main_prob_U} has a \emph{positive}
solution on $\Omega$ with Morse index $k$ and $\rho=\rho_k\to+\infty$ as $k\to+\infty$.
This kind of results justifies the choice of classifying the solutions in terms of their Morse index, rather
than in terms of their nodal properties.
\end{remark}

Motivated by the previous remark, the first question we address in this paper concerns the boundedness of $\setp_k$.
We provide the following complete classification.
\begin{theorem}\label{thm:bbd_index}
For every $\Omega\subset\R^N$ bounded $C^1$ domain, $k\ge1$, $1<p<2^*-1$,
\[
\sup\setp_k(p,\Omega) < +\infty
\qquad\iff\qquad
p \ge 1+\frac{4}{N}.
\]
\end{theorem}
The proof of such result, which is outlined in Section \ref{sec:blow-up}, is obtained
by a detailed blow-up analysis of sequences of solutions with bounded Morse index, via
suitable a priori pointwise estimates (see \cite{MR2063399}). In this respect, the
regularity assumption on $\partial\Omega$ simplifies the treatment of possible
concentration phenomena towards the boundary. The argument, which holds
for solutions which possibly change sign, is inspired by \cite{MR2825606}, where the
case of positive solutions is treated.

Once Theorem \ref{thm:bbd_index} is established, in case $p\geq 1 + 4/N$ two questions arise, namely:
\begin{enumerate}
 \item is it possible to provide lower bounds for $\sup\setp_k$? Is it true that $\sup\setp_k$ is strictly increasing in $k$, or, at least, that $\sup\setp_k > \sup\setp_1$ for some $k$?
 \item is \eqref{eq:main_prob_U} solvable for every $\rho\in(0,\sup\setp_k)$, or at least can we characterize some subinterval of solvability?
\end{enumerate}
It is clear that both issues can be addressed by characterizing
values of $\rho$ for which existence (and multiplicity) of solutions with
bounded Morse index can be guaranteed. To this aim, it can be useful to restate
problem \eqref{eq:main_prob_U} as
\begin{equation}\label{eq:main_prob_u}
\begin{cases}
-\Delta u + \lambda u = \mu|u|^{p-1}u  & \text{in }\Omega,\\
\int_\Omega u^2\,dx = 1, \quad u=0  & \text{on }\partial\Omega,
\end{cases}
\qquad\text{where}\quad
\begin{cases}
U=\sqrt{\rho} u\\
\mu = \rho^{(p-1)/2},
\end{cases}
\end{equation}
where now $\mu>0$ is prescribed. Since
\begin{equation}
\label{Emu}
\text{both } \mathcal{E}_{\mu}(u) := \frac{1}{2}\int_{\om}|\nabla u|^2- \frac{\mu}{p+1}\int_{\om}| u|^{p+1}
\qquad
\text{and }\Mcal=\Mcal_1=\{u : \|u\|_{L^2(\Om)}=1\}
\end{equation}
are invariant under the $\Z_2$-action of the involution $u\mapsto -u$, solutions of
\eqref{eq:main_prob_u} can be found via min-max principles in the framework of index
theories (see e.g. \cite[Ch. II.5]{St_2008}). Notice that in the supercritical case
$\Ecal_\mu$ is not bounded from below on $\Mcal$.  Following \cite{MR3318740}, it can be
convenient to parameterize solutions to \eqref{eq:main_prob_u} with respect to the
$H^1_0$-norm, therefore we introduce the sets
\begin{equation}\label{eq:defBU}
\mathcal{B}_\alpha:=\left\{u\in \Mcal:\,\int_\Omega |\nabla u|^2\,dx<\alpha\right\},\quad\quad
\mathcal{U}_\alpha:=\left\{u\in \Mcal:\,\int_\Omega |\nabla u|^2\,dx=\alpha\right\}.
\end{equation}
Introducing the first Dirichlet eigenvalue of $-\Delta$ in $H^1_0(\Omega)$,
$\lambda_1(\Omega)$, we have that the sets above are non-empty whenever $\alpha>
\lambda_1(\Omega)$.
Since we are interested in critical points
having Morse index bounded from above, following \cite{MR968487,MR954951,MR991264} we
introduce the following notion of genus.
\begin{definition}\label{def:genus}
Let $A\subset H^1_0(\Omega)$ be a closed set, symmetric with respect to the origin (i.e.
$-A=A$). We define the \emph{genus} $\gamma$ of a $A$ as
\[
\gamma(A) := \sup\{m : \exists h \in C(\sphere^{m-1};A),\, h(-u)=-h(u)\}.
\]
Furthermore, we define
\[
\Sigma_{\alpha}=\{A\subset \overline{\mathcal{B}}_\alpha: A\text{ is closed and }-A=A\},
\qquad
\Sigma^{(k)}_{\alpha}=\{A\in \Sigma_{\alpha} : \gamma(A)\ge k\},
\]
\end{definition}
We remark that this notion of genus is different from the
classical one of \emph{Krasnoselskii genus}, which is well suited for estimates
of the Morse index from below, rather than above. Nonetheless, $\gamma$ shares with
the Krasnoselskii genus most of the main properties of an index
\cite{MR0163310,MR0065910}. In particular, by the Borsuk-Ulam Theorem, any set $A$
homeomorphic to the sphere $\sphere^{m-1} := \partial B_1 \subset \R^m$ has genus
$\gamma(A) = m$. Furthermore, we show in Section \ref{sec:2const} that $\Sigma^{(k)}_{\alpha}$ is not empty,
provided $\alpha>\lambda_k(\Omega)$ (the $k$-th Dirichlet eigenvalue of $-\Delta$ in $H^1_0(\Omega)$).

Equipped with this notion of genus we provide two different variational principles
for solutions of \eqref{eq:main_prob_u} (and thus of \eqref{eq:main_prob_U}). The first
one is based on a variational problem with \emph{two constraints}, which was exploited as the
main tool in proving the results in \cite{MR3318740}.
\begin{theorem}\label{thm:genus_2constr}
Let $k\geq1$ and $\alpha>\lambda_{k}(\Om)$. Then
\begin{equation}
\label{maxmin}
M_{\alpha,\,k}:= \sup_{A\in\Sigma^{(k)}_{\alpha}}\inf_{u\in A}\int_{\Omega}|u|^{p+1}
\end{equation}
is achieved on $\Ucal_\alpha$, and there exists a critical point
$u_\alpha\in \Mcal$ such
that, for some $\lambda_\alpha\in\R$ and $\mu_\alpha>0$,
\begin{equation}
\label{lagreq}
\int_\Omega|\nabla u_\alpha|^2 = \alpha\qquad\text{and}\qquad -\Delta u_\alpha+\lambda_\alpha\,u_\alpha=\mu_\alpha |u_\alpha|^{p-1}u_\alpha\quad \text{in }\Omega.
\end{equation}
\end{theorem}
As a matter of fact, the results in \cite{MR3318740} were obtained by a detailed analysis of the map $\alpha \mapsto \mu_\alpha$ in the case $k=1$, i.e.
when dealing with
\[
M_{\alpha,1} = \max\left\{\|u\|_{L^{p+1}}^{p+1} : \|u\|_{L^2}^2=1,\,\|\nabla u\|_{L^2}^2=\alpha \right\}.
\]
In the present paper we do not investigate the properties of the map
$\alpha \mapsto \mu_\alpha$ for general $k$, but we rather  prefer to exploit the
characterization of $M_{\alpha,k}$ in connection with
a second variational principle, which deals with only
\emph{one constraint}.
\begin{theorem}\label{thm:genus_1constr}
Let $1+{N}/{4}\leq p<2^*-1$. There exists a sequence $(\hat \mu_k)_k$
(depending on $\Omega$ and $p$) such that, for every $k\geq 1$ and $0<\mu<\hat \mu_k$, the
value
\begin{equation}
\label{infsuplev}
c_k:= \inf_{A\in\Sigma^{(k)}_{\alpha}}
\sup_{A}\Ecal_\mu,
\end{equation}
is achieved in $\mathcal{B}_\alpha$, for a suitable $\alpha>\lambda_{k}(\Om)$. Furthermore there exists a critical point $u_\mu\in \Mcal$ such
that, for some $\lambda_\mu\in\R$,
\[
-\Delta u_\mu+\lambda_\mu\,u_\mu=\mu |u_\mu|^{p-1}u_\mu\quad \text{in }\Omega,
\]
$\|\nabla u\|_{L^2}^2<\alpha$, and $m(u_\mu)\le k$.
\end{theorem}
\begin{remark}
Of course, if $p<1+4/N$, the above theorem holds with $\hat\mu_k=+\infty$ for every $k$.
\end{remark}
\begin{corollary}
Let $\hat \rho_k := \hat \mu_k^{2/(p-1)}$. Then
\[
(0,\hat\rho_k) \subset \setp_k.
\]
\end{corollary}
The link between Theorem \ref{thm:genus_2constr} and Theorem \ref{thm:genus_1constr} is that we can provide explicit estimates
of $\hat \mu_k$ (and hence of $\hat\rho_k$) in terms of the map $\alpha\mapsto M_{\alpha,k}$ (see Section \ref{sec:1const}).

We stress that the above
results hold for any Lipschitz $\Omega$. As a first consequence, this allows to extend the
existence result in \cite{MR3318740} to non-radial domains.
\begin{theorem}\label{thm:intro_GS}
For every $0<\rho<\hat\rho_1=\hat\rho_1(\Omega,p)$ problem \eqref{eq:main_prob_U} admits a solution which is a local minimum of
the energy $\Ecal$ on $\Mcal_\rho$. In particular, $U$ is positive, has Morse index one and the associated solitary wave is orbitally
stable.

Furthermore, for every Lipschitz $\Omega$,
\begin{itemize}
\item $\displaystyle 1<p<1+\frac{4}{N} \implies \hat\rho_1\left(\Omega,p\right)
= +\infty$,
\item $\displaystyle p=1+\frac{4}{N} \implies \hat\rho_1\left(\Omega,p\right)
\geq  \|Z_{N,p}\|^2_{L^2(\R^N)}$,
\item $\displaystyle 1+\frac{4}{N}<p<2^*-1 \implies \hat\rho_1\left(\Omega,p\right)
\geq D_{N,p} \lambda_1(\Omega)^{\frac{2}{p-1}-\frac{N}{2}}$,
\end{itemize}
where the universal constant $D_{N,p}$ is explicitly written in terms of $N$ and $p$ in
Section \ref{sec:1const}.
\end{theorem}
\begin{remark}\label{rem:introGS}
Of course, in the subcritical and critical cases, $c_1$ is actually a global minimum.
Furthermore, the lower bound for the supercritical case agrees with that of the critical one
since, as shown in Section  \ref{sec:1const}, $D_{N,1+4/N} = \|Z_{N,p}\|^2_{L^2(\R^N)}$
(and $\lambda_1(\Omega)$ is raised to the $0^{\text{th}}$-power). Notice that the estimate for the supercritical case is new also in the case $\Omega=B_1$.
\end{remark}
We observe that the exponent of $\lambda_1(\Omega)$ in the supercritical threshold is negative,
therefore such threshold decreases with the size of $\Omega$.

Once the first thresholds have been estimated, we turn to the higher ones: by exploiting
the relations between $M_{\alpha,k}$ and $c_k$, we can show that the thresholds obtained
for Morse index one--solutions in Theorem \ref{thm:intro_GS} can be increased, by
considering higher Morse index--solutions, at least for some exponent.
\begin{proposition}\label{thm:intro_3>1}
For every $\Omega$ and $1<p<2^*-1$,
\[
\hat\rho_3\left(\Omega,p\right) \geq 2 \cdot D_{N,p} \lambda_3(\Omega)^{\frac{2}{p-1}-\frac{N}{2}}.
\]
\end{proposition}
\begin{remark}
In the critical case, the lower bound for $\hat\rho_3$ provided by Proposition \ref{thm:intro_3>1} is twice that for $\hat\rho_1$ obtained in Theorem \ref{thm:intro_GS}.
By continuity, the estimate for $\hat\rho_3$ is larger than that for $\hat\rho_1$ also when
$p$ is supercritical, but not too large. To quantify such assertion, we can use Yang's inequality
\cite{MR1894540,MR2262780}, which implies that for every $\Omega$ it holds
\[
\lambda_3(\Omega)\leq \left(1+\frac{N}{4}\right)2^{2/N} \lambda_1(\Omega).
\]
We deduce that $2 \cdot D_{N,p} \lambda_3(\Omega)^{\frac{2}{p-1}-\frac{N}{2}}
\geq  D_{N,p} \lambda_1(\Omega)^{\frac{2}{p-1}-\frac{N}{2}}$ whenever
\[
p\leq 1+\frac{4}{N} + \frac{8}{N^2\log_2\left(1+\frac{4}{N}\right)}.
\]
In particular, the physically relevant case $N=3$, $p=3$ is covered. Furthermore, if
$N\geq 7$, the above condition holds for every $p<2^*-1$.
\end{remark}
Beyond existence results for
\eqref{eq:main_prob_U}, also multiplicity results can be achieved. A first general
consideration, with this respect, is that Theorem \ref{thm:genus_1constr} holds true
also when using the standard Krasnoselskii genus instead of $\gamma$; this allows to
obtain critical points having Morse index bounded from below (see
\cite{MR968487,MR954951,MR991264}), and therefore to obtain infinitely many solutions,
at least when $\rho$ is less than some threshold. More specifically, we can also
prove the existence of a second solution in the supercritical case, thus extending to
any $\Omega$ the multiplicity result obtained in \cite{MR3318740} for the ball.
Indeed, on the one
hand, in the supercritical case $\Ecal_\mu$ is unbounded from below; on the other hand
the solution obtained in Theorem \ref{thm:genus_1constr}, for $k=1$, is a local minimum.
Thus the Mountain Pass Theorem \cite{MR0370183} applies on $\mathcal{M}$, and a second solution can be
found for $\mu<\hat\mu_1$, see Proposition \ref{mpcritlev} for further details (and also
Remark \ref{rem:further_crit_lev} for an analogous construction for $k\ge2$).

To conclude this introduction, let us mention that the explicit lower bounds obtained in
Theorem \ref{thm:intro_GS} can be easily applied in order to gain
much more information also in the case of special domains, as those considered in
Remark \ref{rem:specialdomains}. For instance, we can prove then following.
\begin{theorem}\label{pro:symm}
Let $\Omega=B$ be a ball in $\R^N$. Then
\[
p<1+\frac{4}{N-1}
\quad\implies\quad
\text{\eqref{eq:main_prob_U} admits a solution for every }\rho>0.
\]
An analogous result holds when $\Omega=R$ is a rectangle, without further restrictions on
$p<2^*-1$.
\end{theorem}
Therefore our starting problem in $\Omega=B$ can be solved for any mass value also in the critical and
supercritical regime, at least for $p$ smaller than this further critical exponent $1+4/(N-1) > 1+ 4/N$.
Of course, higher masses require higher Morse index--solutions. In particular, since by \cite{MR3318740}
we know that $\setp_1(B,1+4/N) = (0,\|Z_{N,p}\|_{L^2})$, we have that for
larger masses, even though no positive solution exists, nodal solutions with higher
Morse index can be obtained: in such cases \eqref{eq:main_prob_U} admits \emph{nodal ground
states with higher Morse index}.

The paper is structured as follows: in Section \ref{sec:blow-up} we perform a blow-up analysis of
solutions with bounded Morse index, in order to prove Theorem \ref{thm:bbd_index}; Section \ref{sec:2const}
is devoted to the analysis of the variational problem with two constraints \eqref{maxmin}
and to the proof of Theorem \ref{thm:genus_2constr};
that of Theorems \ref{thm:genus_1constr}, \ref{thm:intro_GS} and Proposition
\ref{thm:intro_3>1} is developed
in Section \ref{sec:1const}, by means of the variational problem with one constraint
\eqref{infsuplev}; finally, Section
\ref{sec:symm} contains the proof of Theorem \ref{pro:symm}.

\textbf{Notation.} We use the standard notation
$\{\varphi_k\}_{k\geq1}$ for a basis of eigenfunctions of the Dirichlet laplacian in $\Omega$,
orthogonal in $H^1_0(\Omega)$ and orthonormal in $L^2(\Omega)$. Such functions are ordered in such
a way that the corresponding eigenvalues $\lambda_k(\Omega)$ satisfy
\[
0<\lambda_1(\Omega)<\lambda_2(\Omega)\leq\lambda_3(\Omega)\leq\dots,
\]
and $\varphi_1$ is chosen to be positive on $\Omega$. $C_{N,p}$ denotes the universal constant in
the Gagliardo-Nirenberg inequality \eqref{sobest}, which is achieved (uniquely, up to translations and dilations)
by the positive, radially symmetric function $Z_{N,p}\in H^1(\R^N)$, with
\[
\|Z_{N,p}\|^2_{L^2(\R^N)}=\left(\frac{p+1}{2C_{N,p}}\right)^{N/2}.
\]
Finally, $C$ denotes every (positive) constant we need not to specify, whose value may change also within the same formula.

\section{Blow-up analysis of solutions with bounded Morse index}\label{sec:blow-up}

Throughout this section we will deal with a sequence $\{(u_n,\mu_n,\lambda_n)\}_n \subset H^1_0(\Omega)\times\R^+\times\R$ satisfying
\begin{equation}\label{eq:auxiliary_n}
-\Delta u_n+\lambda_n u_n=\mu_n |u_n|^{p-1}u_n,\qquad\int_\Omega u_n^2\, dx=1,\qquad \int_\Omega |\nabla u_n|^2\, dx=:\alpha_n.
\end{equation}
To start with, we recall the following result (actually, in \cite{MR3318740}, the result is stated for positive solution, but the proof
does not require such assumption).
\begin{lemma}[{\cite[Lemma 2.5]{MR3318740}}]\label{lemma:case_alpha_n_bounded}
Take a sequence $\{(u_n,\mu_n,\lambda_n)\}_n$ as in \eqref{eq:auxiliary_n}. Then
\[
\{\alpha_n\}_n \text{ bounded}
\qquad\implies\qquad
\{\lambda_n\}_n,\,\{\mu_n\}_n\text{ bounded}.
\]
\end{lemma}
Next we turn to the study of sequences having arbitrarily large $H^1_0$-norm. In particular, we will focus on sequences of solutions having a common upper
bound on the Morse index
\[
m(u_n) = \max\left\{k : \begin{array}{l}
\exists V\subset H^1_0(\Omega),\,\dim(V)= k:\forall v\in V\setminus\{0\}\smallskip\\
\displaystyle\int_\Omega |\nabla v|^2 + \lambda_n v^2 - p\mu_n|u_n|^{p-1}v^2\,dx<0
\end{array}
\right\}.
\]
Throughout this section we will assume that
\begin{equation}\label{eq:mainass_secMorse}
\text{the sequence }\{(u_n,\mu_n,\lambda_n)\}_n\text{ satisfies \eqref{eq:auxiliary_n},
with }\alpha_n\to+\infty\text{ and }m(u_n)\leq \bar k,
\end{equation}
for some $\bar k\in\N$ not depending on $n$.
\begin{lemma}\label{lem:lambda_bdd_below}
Let \eqref{eq:mainass_secMorse} hold. Then
\(
\lambda_n \geq -\lambda_{\bar k}(\Omega).
\)
\end{lemma}
\begin{proof}
Assume, to the contrary, that for some $n$ it holds $\lambda_n < -\lambda_{\bar k}(\Omega)$. For any real $t_1,\dots t_{\bar k}$
we define
\[
\phi := \sum_{h=1}^{\bar k} t_h \varphi_h.
\]
By denoting $J_{\lambda,\mu}(u)=\Ecal_\mu(u)+\frac{\lambda}{2}\|u\|_{L^2}^2$, so that
Morse index properties can be written in terms of $J''_{\lambda,\mu}$, we have
\[
\begin{split}
J''_{\lambda_n,\mu_n}(u_n)[u_n,\phi] &=  -(p-1)\mu_n\int_\Omega |u_n|^{p-1}u_n\phi,\\
J''_{\lambda_n,\mu_n}(u_n)[\phi,\phi] &= \sum_{h=1}^{\bar k} t_h^2 \int_\Omega \bigl (|\nabla \varphi_h| + \lambda_n \varphi_h^2\bigr )\,dx   - p\mu_n\int_\Omega |u_n|^{p-1}\phi^2\,dx \\
  &\leq \sum_{h=1}^{\bar k} t_h^2(\lambda_{h}(\Omega) + \lambda_n)  - (p-1)\mu_n\int_\Omega |u_n|^{p-1}\phi^2\,dx \leq  - (p-1)\mu_n\int_\Omega |u_n|^{p-1}\phi^2\,dx,
\end{split}
\]
where equality holds if and only if $t_1=\dots=t_{\bar k}=0$. As a consequence
\begin{multline*}
J''_{\lambda_n,\mu_n}(u_n)[t_0 u_n+ \phi, t_0 u_n + \phi] \leq  -t_0^2 (p-1)\mu_n\int_\Omega |u_n|^{p-1}u_n^2 \\
- 2t_0(p-1)\mu_n\int_\Omega |u_n|^{p-1}u_n\phi \,dx  - (p-1)\mu_n\int_\Omega |u_n|^{p-1}\phi^2\,dx.
\end{multline*}
We deduce that $J''_{\lambda_n,\mu_n}(u_n)$ is negative definite on $\spann\{u_n, \varphi_1,\dots,\varphi_{\bar k}\}$, in contradiction with the bound on the Morse index (note that
$u_n$ cannot be a linear combination of a finite number of eigenfunctions, otherwise using the equations we would obtain that such eigenfunctions are linearly dependent).
\end{proof}
\begin{lemma}
\label{localblow}
Let \eqref{eq:mainass_secMorse} hold.
Then $\lambda_n\to +\infty$.
\end{lemma}
\begin{proof}
By Lemma \ref{lem:lambda_bdd_below} we have that $\lambda_n$ is bounded below. As a consequence, we can use H\"{o}lder inequality with $\|u_n\|_{L^2}=1$ and \eqref{eq:auxiliary_n} to write
\[
\mu_n\,\|u_n\|^{p-1}_{L^{\infty}}\ge \mu_n \,\|u_n\|^{p+1}_{L^{p+1}}=\alpha_n+\lambda_n
\rightarrow +\infty.
\]
Let us define
\begin{equation}
\label{equn}
U_n: =\mu_n^{\frac{1}{p-1}}\,u_n,
\quad\text{ so that }
-\Delta U_n+\lambda_n U_n=|U_n|^{p-1}U_n\quad \text{in}\,\,\Omega,\quad
U|_{\partial\Omega}=0.
\end{equation}
Pick $P_n\in\Om$ such that $|U_n(P_n)|=\|U_n\|_{L^{\infty}(\Omega)}$ and set
\begin{equation}
\label{tildepsn}
\tilde\eps_n: =|U_n(P_n)|^{-\frac{p-1}{2}}=\frac{1}{\sqrt{\mu_n\,\|u_n\|^{p-1}_{L^{\infty}}}}\longrightarrow 0
\end{equation}
Hence, $|U_n(P_n)|\to+\infty$; moreover, as $P_n$ is a point of positive maximum or of negative minimum, we have
\begin{equation}
\nonumber
0\le \frac{-\Delta U_n(P_n)}{U_n(P_n)}=|U_n(P_n)|^{p-1}-\lambda_n\,.
\end{equation}
Thus $\lambda_n|U_n(P_n)|^{1-p}\le 1$, and since $\lambda_n$ is bounded from below, we conclude
\begin{equation}
\label{limtildelam}
\frac{\lambda_n}{|U_n(P_n)|^{p-1}}\longrightarrow \tilde\lambda\in [0,1].
\end{equation}
Now, we are left to prove that $\tilde\lambda>0$. Let us define
\begin{equation}
\label{tildeVn}
\tilde V_n(y)=\tilde\eps_n^{\frac{2}{p-1}}\, U_n(\tilde\eps_n\,y+P_n),\quad\quad y\in \tilde\Om_n
:=\big (\Om-P_n\big )/\tilde\eps_n,
\end{equation}
and let $d_n := d(P_n,\partial\Om)$; we have, up to subsequences,
\[
\frac{\tilde\eps_n}{d_n}\longrightarrow L\in [0,+\infty]
\qquad\text{and}\qquad
\tilde\Om_n\rightarrow\left\{
                \begin{array}{ll}
                  \R^n, & \text{if $L=0$;} \\
                  H, & \text{if $L>0$,}
                \end{array}
              \right.
\]
where $H$ is a half-space such that $0\in \overline H$ and $d(0,\partial H)=1/L$.
The function $\tilde V_n$ satisfies
\begin{equation}
\nonumber
\left\{
  \begin{array}{ll}
    -\Delta \tilde V_n+\lambda_n\,\tilde \eps_n^2\,\tilde V_n=|\tilde V_n|^{p-1}\tilde V_n, & \hbox{in}\,\, \tilde\Om_n;\\
    |\tilde V_n|\le |\tilde V_n(0)|=1, & \hbox{in}\,\, \tilde\Om_n;\\
    \tilde V_n=0, & \hbox{on}\,\, \partial\tilde \Om_n.
  \end{array}
\right.
\end{equation}
From \eqref{tildepsn} and \eqref{limtildelam} we get $\tilde\eps_n^2\,\lambda_n\rightarrow \tilde\lambda$; hence, by elliptic regularity and up to a further subsequence, $\tilde V_n\rightarrow \tilde V$ in $\mathcal{C}^1_{\loc}(\overline H)$ where $\tilde V$ solves
\begin{equation}
\label{limprob1}
\left\{
  \begin{array}{ll}
    -\Delta \tilde V+\tilde\lambda\,\tilde V=|\tilde V|^{p-1}\tilde V, & \hbox{in}\,\, H;\\
    |\tilde V|\le |\tilde V(0)|=1, & \hbox{in}\,\, H;\\
    \tilde V=0, & \hbox{on}\,\, \partial H.
  \end{array}
\right.
\end{equation}
Since $\sup_n m(U_n)\leq \bar k$ (as a solution to \eqref{equn}), one can show as in Theorem $3.1$ of \cite{MR2825606} that $m(\tilde V)\leq \bar k$. In particular, $\tilde V$ is stable outside a compact set (see Definition $2.1$ in \cite{MR2825606}) so that, by Theorem $2.3$ and Remark $2.4$  of \cite{MR2825606}, we have
$$\tilde V(x)\rightarrow 0 \quad\quad \text{as} \quad\quad |x|\rightarrow +\infty.$$
Moreover, since $\tilde V$ is not trivial, we also have that $\tilde\lambda>0$. For, if $\tilde\lambda=0$ the function
$\tilde V$ would be a  solution of the Lane-Emden equation
$-\Delta u=|u|^{p-1}u$ either in $\R^n$ or in $H$. In both cases, $\tilde V$ would contradict Theorems $2$ and $9$ of \cite{MR2322150}, being non trivial and stable outside a compact set. Thus, $\tilde\lambda >0$ and by \eqref{limtildelam} we conclude $\lambda_n\rightarrow +\infty$.
\end{proof}
\begin{remark}\label{rem4}
We stress that the scaling argument in Lemma \ref{localblow}, leading to the limit problem \eqref{limprob1} (with $\tilde\lambda>0$), can be repeated also near points of \emph{local} extremum.
More precisely, let $Q_n$ be such that $|U_n(Q_n)|\to +\infty$ and
$$
|U_n(Q_n)|=\max_{\Om\cap B_{R_n\tilde\eps_n}(Q_n)}U_n,
$$
for some $R_n\to +\infty$. Then the above procedure can be repeated by replacing $P_n$ with $Q_n$ in definition \eqref{tildepsn}.
\end{remark}

The local description of the asymptotic behaviour of the solutions $U_n$ to \eqref{equn} with bounded Morse index can be carried out more conveniently by defining
the sequence (see \cite[Theorem $3.1$]{MR2825606})
\begin{equation}
\label{defVn}
V_n(y)=\eps_n^{\frac{2}{p-1}}\, U_n(\eps_n\,y+P_n),\quad y\in \Om_n
:=\frac{\Om-P_n}{\eps_n},
\end{equation}
where $P_n$ is defined before \eqref{tildepsn}, and $\eps_n=\frac{1}{\sqrt{\lambda_n}}\to 0$.
Then, $V_n$ satisfies
\begin{equation}
\nonumber
\left\{
  \begin{array}{ll}
    -\Delta V_n+ V_n=| V_n|^{p-1} V_n, & \hbox{in}\,\,\Om_n;\\
    |V_n|\le | V_n(0)|=\big ({\eps_n/\tilde\eps_n}\big )^{\frac{2}{p-1}}\rightarrow
\tilde\lambda^{-\frac{1}{p-1}}, & \hbox{in}\,\, \Om_n;\\
    V_n=0, & \hbox{on}\,\, \partial\Om_n.
  \end{array}
\right.
\end{equation}
As before, we have (up to a subsequence) $V_n\rightarrow  V$ in $\mathcal{C}^1_{\mathrm{loc}}(\overline H)$ where $H$ is either $\R^N$ or a half space and $V$ solves
\begin{equation}
\label{limprob2}
\left\{
  \begin{array}{ll}
    -\Delta  V+ V=| V|^{p-1} V, & \hbox{in}\,\, H;\\
    |V|\le | V(0)|=\tilde\lambda^{-\frac{1}{p-1}}, & \hbox{in}\,\, H;\\
     V=0, & \hbox{on}\,\, \partial H.
  \end{array}
\right.
\end{equation}
By recalling the discussion following \eqref{limprob1} we also have $m(V)<+\infty$.
We collect some well known property of such a $V$ in the following result.
\begin{theorem}[\cite{MR688279,MR2825606,MR2322150,MR2785899}]\label{thm:unif_est_Farina}
Let $V$ be a classical solution to \eqref{limprob2} such that $m(V)\leq\bar k$. Then:
\begin{enumerate}
 \item $H=\R^N$;
 \item $V(x)\to 0$ as $|x|\rightarrow +\infty$, $V \in H^1(\R^N)\cap L^{p+1}(\R^N)$;
 \item there exist $C$ only depending on $\bar k$ (and not on $V$) such that
 \[
 \|V\|_{L^{\infty}} + \|\nabla V\|_{L^{\infty}}<C.
 \]
\end{enumerate}
\end{theorem}
\begin{proof}
Claim 2 follows from Theorem 2.3 and Remark 2.4 of \cite{MR2825606}, see also \cite[Remark 1.4]{MR688279}. As a consequence,
Theorem 1.1 of \cite[Remark 1.4]{MR688279} readily applies, providing claim 1 ($V$ is not trivial as $V(0)>0$). On the other hand, the $L^\infty$ estimates in claim 3.
are proved in Theorem 1.9 of \cite{MR2785899}.
\end{proof}
\begin{corollary}
\label{distpnbound}
If the sequence $\{U_n\}$ of  solutions to \eqref{equn} has uniformly bounded Morse index, and if $P_n\in\Om$ is such that $|U_n(P_n)|=\|U_n\|_{L^{\infty}(\Omega)}\to+\infty$, then
$$
\sqrt{\lambda_n}\,d(P_n,\partial\Om)\rightarrow +\infty,\qquad\text{where }\frac{\lambda_n}{|U_n(P_n)|^{p-1}}\to\tilde\lambda\in(0,1].
$$
\end{corollary}
\begin{remark}
Recall that $Z_{N,p}$, the unique positive solution to $-\Delta  u+ u=| u|^{p-1} u$ in $\R^N$, has Morse index $1$ \cite{MR969899}; then, if $V$ solves \eqref{limprob2} in $\R^N$
and $1<m(V)<+\infty$, then $V$ is necessarily sign-changing.
\end{remark}
Following the same pattern as in \cite{MR2825606}, we now analyze the global behaviour of a sequence $\{U_n\}$ of  solutions to \eqref{equn} for $\lambda_n\to +\infty$, assuming that
\(
\lim_{n\to +\infty} m(U_n)\leq\bar k<\infty.
\)

By the previous discussion, if $P^1_n$ is a sequence of points such that $|U_n(P^1_n)|=\|U_n\|_{L^{\infty}(\Omega)}$, we have
$|U_n(P^1_n)|\rightarrow +\infty$ and ${\lambda_n}\,d(P^1_n,\partial\Om)^2\rightarrow +\infty$. We now look for other possible sequences of (local) extremum points $P^i_n$, $i=2,3,..$, along which $|U_n|$ goes to infinity. For any $R>0$, consider the quantity
\begin{equation}
\nonumber
h_1(R)=\limsup_{n\to +\infty} \Bigl (\lambda_n^{-\frac{1}{p-1}}\max_{|x-P^1_n|\ge R\,\lambda_n^{-1/2}}
|U_n(x)| \Bigr ).
\end{equation}
We will prove that if $h_1(R)$ is \emph{not vanishing} for large $R$, then there exists a 'blow-up' sequence $P^2_n$ for $u_n$, 'disjoint' from $P^1_n$. Indeed, let us suppose that
\begin{equation}
\nonumber
\limsup_{R\to +\infty} h_1(R)=4\delta>0.
\end{equation}
Hence, up to a subsequence and for arbitrarily large $R$, we have
\begin{equation}
\label{ass1}
\lambda_n^{-\frac{1}{p-1}}\max_{|x-P^1_n|\ge R\,\lambda_n^{-1/2}}
|U_n(x)| \ge 2\delta.
\end{equation}
Since $U_n$ vanishes on $\partial\Om$, there exists
$P^2_n\in\Om\backslash B_{R\,\lambda_n^{-1/2}}(P_n^1)$
such that
\begin{equation}
\label{pn2}
|U_n(P_n^2)|=\max_{|x-P^1_n|\ge R\,\lambda_n^{-1/2}}|U_n(x)|.
\end{equation}
Clearly, assumption \eqref{ass1} implies that $|U_n(P_n^2)|\rightarrow +\infty$. We first prove that the sequences $P_n^1$ and $P_n^2$ are far away each other.
\begin{lemma}
\label{disj}
Take $R$ such that \eqref{ass1} holds, and let $P_n^2$ be defined as in \eqref{pn2}; then
\begin{equation}
\label{limp1p2}
\lambda_n^{1/2}|P_n^2-P^1_n|\rightarrow +\infty
\end{equation}
as $n\to \infty$.
\end{lemma}
\begin{proof}
Assuming the contrary one would get, up to a subsequence
\[
\lambda_n^{1/2}|P_n^2-P^1_n|\rightarrow R'\ge R.
\]
Let us now recall
that by \eqref{defVn} and the subsequent discussion, we have:
\begin{equation}
\label{limblowseq}
\lambda_n^{-\frac{1}{p-1}}\, U_n(\lambda_n^{-1/2}\,y+P^1_n) =: V^1_n(y)\rightarrow V(y)\quad \textrm{in}
\,\, \mathcal{C}^1_{\loc}(\R^N)
\end{equation}
as $n\to +\infty$. Then, up to subsequences,
\begin{equation}
\nonumber
\lambda_n^{-\frac{1}{p-1}}\,|U_n(P_n^2)|=\big |V^1_n\bigr(\lambda_n^{1/2}(P_n^2-P^1_n)\bigl)\big|
\rightarrow \big |V(y')\big |,\quad |y'|=R'\ge R.
\end{equation}
Since $V$ is vanishing for $|y|\to +\infty$, one can choose $R$ such that
$|V(y)|\le\delta$ for every $ |y|\ge R$.
But this contradicts \eqref{ass1}.
\end{proof}
Furthermore, we also have that the blow-up points stay far away from the boundary.
\begin{lemma}
\label{distbd}
Assume \eqref{ass1} and let $P_n^2$ be defined as in \eqref{pn2}; then
\begin{equation}
\label{distp2nbound}
\sqrt{\lambda_n}\,d(P^2_n,\partial\Om)\rightarrow +\infty
\end{equation}
as $n\to \infty$. Moreover,
\begin{equation}
\label{maxp2nball}
|U_n(P_n^2)|=\max_{\Om\cap B_{R_n\lambda^{-1/2}_n}(P^2_n)}|U_n|
\end{equation}
for some $R_n\to +\infty$.
\end{lemma}
\begin{proof}
Let us set
\begin{equation}
\nonumber
\tilde\eps^2_n: =|U_n(P^2_n)|^{-\frac{p-1}{2}}\quad \mathrm{and}
\quad R_n^{(2)}:=\frac{1}{2}\,\frac{|P_n^2-P_n^1|}{\tilde\eps^2_n}.
\end{equation}
Clearly, $\tilde\eps^2_n\rightarrow 0$; moreover, by \eqref{ass1} and \eqref{pn2},
$\tilde\eps^2_n\le (2\delta)^{-\frac{p-1}{2}}\lambda_n^{-1/2}$, so that
$$R^{(2)}_n\ge  \frac{(2\delta)^{\frac{p-1}{2}}}{2}\,\lambda_n^{1/2}\,{|P_n^2-P_n^1|}
\rightarrow +\infty, $$
as $n\to +\infty$ by Lemma \ref{disj}. We claim that this implies
\begin{equation}
\label{maxp2n}
|U_n(P_n^2)|=\max_{\Om\cap B_{R^{(2)}_n\tilde\eps^2_n}(P^2_n)}| U_n|.
\end{equation}
For, if $x\in B_{R^{(2)}_n\tilde\eps^2_n}(P^2_n)$, by \eqref{limp1p2} we would have
$$|x-P^1_n|\ge |P^2_n-P^1_n|-|x-P^2_n|\ge \frac{1}{2}\,|P^2_n-P^1_n|\ge R\,\lambda_n^{-1/2},$$
for arbitrarily large $R$. This means that
$$\Om\cap B_{R^{(2)}_n\tilde\eps^2_n}(P^2_n)\subset \Om\backslash B_{R\,\lambda_n^{-1/2}}(P^1_n).$$
Then, the claim follows. Now, by recalling Remark \ref{rem4}, we can apply to $U_n$ satisfying \eqref{maxp2n} the same scaling arguments as in the proof of Lemma \ref{localblow}, so that we conclude
$$
0< \lim_{n\to +\infty}\tilde\eps_n^2\,\sqrt{\lambda_n}.
$$
Hence,  \eqref{maxp2nball} holds by defining $R_n=R^{(2)}_n\tilde\eps_n^2\,\sqrt{\lambda_n}$,
and  \eqref{distp2nbound} follows by Corollary  \ref{distpnbound}.
\end{proof}
We can now iterate the previous arguments: let us define, for  $k\ge 1$,
\begin{equation}
\label{defhn}
h_k(R)=\limsup_{n\to +\infty} \Bigl (\lambda_n^{-\frac{1}{p-1}}\max_{d_{n,k}(x)\ge R\,\lambda_n^{-1/2}}
|U_n(x)| \Bigr ),
\end{equation}
where
$$d_{n,k}(x): =\min\{|x-P^i_n|\,:\, i=1,...,k\}$$
and the sequences $P^i_n$ are such that
\begin{equation}
\nonumber
\sqrt{\lambda_n}\,d(P^i_n,\partial\Om)\rightarrow +\infty;\quad \lambda_n^{1/2}|P_n^i-P^j_n|\rightarrow +\infty,\quad\quad i,j=1,...,k,\quad i\neq j
\end{equation}
as $n\to +\infty$.
Assume that
$$\limsup_{n\to +\infty} h_k(R)=4\delta>0.$$
As before, up to a subsequence and for arbitrarily large $R$, we have
\begin{equation}
\label{assk}
\lambda_n^{-\frac{1}{p-1}}\max_{d_{n,k}(x)\ge R\,\lambda_n^{-1/2}}
|U_n(x)| \ge 2\delta
\end{equation}
and there exist $P^{k+1}_n$ so that
\begin{equation}
\nonumber
|U_n(P_n^{k+1})|=\max_{d_{n,k}(x)\ge R\,\lambda_n^{-1/2}}|U_n(x)|
\end{equation}
with $\lim_{n\to +\infty}|U_n(P_n^{k+1})|=+\infty$.
Moreover, as in Lemma \ref{disj} we deduce that, for every $i=1,...,k$
\begin{equation}
\label{limblowseqi}
\lambda_n^{-\frac{1}{p-1}}\,U_n(\lambda_n^{-1/2}\,y+P^i_n): = V^i_n(y)\rightarrow V^i(y)\quad \textrm{in}
\,\, \mathcal{C}^1_{\loc}(\R^N)
\end{equation}
as $n\to +\infty$; hence, by \eqref{assk} and again from the vanishing of $V$ at infinity, we conclude that
\begin{equation}
\lambda_n^{1/2}|P_n^{k+1}-P^i_n|\rightarrow +\infty
\end{equation}
as $n\to \infty$, for every $i=1,...,k$. Setting now
\begin{equation}
\nonumber
\tilde\eps^{k+1}_n: =|U_n(P^{k+1}_n)|^{-\frac{p-1}{2}}\quad \mathrm{and}
\quad R_n^{(k+1)}:=\frac{1}{2}\,\frac{d_{n,k}(P^{k+1}_n)}{\tilde\eps^{k+1}_n}
\end{equation}
we still have $\tilde\eps^{k+1}_n\to 0$ and, by \eqref{assk}, $R_n^{(k+1)}  \to +\infty$  as $n\to \infty$ (see Lemma \ref{distbd}). Then, by the same arguments as in Lemma \ref{distbd}, we get
\begin{equation}
\label{maxpkn}
|U_n(P_n^{k+1})|=\max_{\Om\cap B_{R^{(k+1)}_n\tilde\eps^{k+1}_n}(P^{k+1}_n)} |u_n|\,,
\end{equation}
and furthermore
$$ \lim_{n\to +\infty}\tilde\eps_n^{k+1}\,\sqrt{\lambda_n}>0\,,$$
so that by defining $R_n=:R^{(k+1)}_n\tilde\eps_n^{k+1}\,\sqrt{\lambda_n}\rightarrow +\infty$
we have
\begin{equation}
\label{maxpknball}
|U_n(P_n^{k+1})|=\max_{\Om\cap B_{R_n\lambda^{-1/2}_n}(P^{k+1}_n)}| U_n|.
\end{equation}
Now, by the same arguments as in \cite{MR2825606}, it turns out that the iterative procedure must stop after \emph{at most} $\bar k-1$ steps, where
$\bar k =\lim_{n\to +\infty} m(u_n)$. Thus, we have proved:
\begin{proposition}
\label{glob1}
Let $\{U_n\}_n$ be a solution sequence to \eqref{equn} such that $\lambda_n\to+\infty$
and $m(U_n)\leq\bar k$.
Then, up to a subsequence, there exist $P_n^1,...,P_n^k$, with $k\le \bar k$ such that
\begin{equation}
\label{limpin}
\sqrt{\lambda_n}\,d(P^i_n,\partial\Om)\rightarrow +\infty;\quad \lambda_n^{1/2}|P_n^i-P^j_n|\rightarrow +\infty,\quad\quad i,j=1,...,k,\quad i\neq j
\end{equation}
as $n\to +\infty$ and
\begin{equation}
\nonumber
|U_n(P_n^{i})|=\max_{\Om\cap B_{R_n\lambda^{-1/2}_n}(P^{i}_n)}|U_n|,\quad i=1,...,k,
\end{equation}
for some $R_n\to +\infty$ as $n\to +\infty$. Finally,
\begin{equation}
\label{limhr0}
\lim_{R\to +\infty} h_k(R)=0
\end{equation}
where $h_k(R)$ is given by \eqref{defhn}.
\end{proposition}
We now show that the sequence $U_n$ decays exponentially away from the blow-up points.
\begin{proposition}
\label{glob2}
Let $\{U_n\}_n$ satisfy the assumptions of Proposition \ref{glob1}. Then, there exist
$P_n^1,...,P_n^k$ and
positive constants $C$, $\gamma$, such that
\begin{equation}
\label{stimglob}
|U_n(x)|\le C\lambda^{\frac{1}{p-1}}_n \sum_{i=1}^ke^{-\gamma\sqrt{\lambda_n}|x-P_n^i|}\,,\quad\quad \forall \,x\in\Om,\quad n\in\N\,.
\end{equation}
\end{proposition}
\begin{proof}
By \eqref{limhr0}, for large $R>0$ and $n>n_0(R)$ it holds
\begin{equation}
\nonumber
\lambda_n^{-\frac{1}{p-1}}\max_{d_{n,k}(x)\ge R\,\lambda_n^{-1/2}}
|U_n(x)| \le  \Bigr (\frac{1}{2p} \Bigl )^{\frac{1}{p-1} }
\end{equation}
Then, for $n>n_0(R)$ and for $x\in \{d_{n,k}(x)\ge R\,\lambda_n^{-1/2}\}$, we have
\begin{equation}
\nonumber
a_n(x): = \lambda_n-p |U_n(x)|^{p-1}\ge \lambda_n-\frac{\lambda_n}{2}=\frac{\lambda_n}{2}
\end{equation}
We stress that the linear operator
\begin{equation}
\nonumber
L_n: =-\Delta + a_n(x)
\end{equation}
comes from the linearization of equation \eqref{equn} at $U_n$; let us compute this operator on the functions
$$\phi^i_n(x)=e^{-\gamma\sqrt{\lambda_n}\,|x-P^i_n|}\,,\quad\quad \gamma>0,\quad\quad i=1,...,k$$
in $\{d_{n,k}(x)\ge R\,\lambda_n^{-1/2}\}$. We obtain:
$$L_n \phi^i_n(x)=\lambda_n\phi^i_n(x)\Bigr [-\gamma^2+(N-1)\frac{\gamma}{\sqrt{\lambda_n}\,|x-P^i_n|}
+\frac{a_n(x)}{\lambda_n}\Bigl ]\ge \lambda_n\phi^i_n(x)\bigr [-\gamma^2+1/2\bigl ]\ge 0$$
for $n$ large, provided $0<\gamma\le 1/\sqrt 2$. Moreover, for $|x-P^i_n|=R\lambda_n^{-1/2}$, $ i=1,...,k,$ and $R$ large we have
$$
e^{\gamma R}\phi^i_n(x)-\lambda_n^{-\frac{1}{p-1}}|U_n(x)|=
1-\lambda_n^{-\frac{1}{p-1}}|U_n(x)|>0
$$
as $n\to +\infty$, by \eqref{limblowseq}. Note further that
$$\{x:d_{n,k}(x)= R\,\lambda_n^{-1/2}\} = \bigcup_{i=1}^k \partial B_{R\,\lambda_n^{-1/2}}(P_n^i)
\subset \Om$$
for large enough $n$. Then, by defining
\begin{equation}
\nonumber
\phi_n: = e^{\gamma R}\lambda_n^{\frac{1}{p-1}}\sum_{i=1}^k \,\phi^i_n
\end{equation}
we have
$$\phi_n(x)-|U_n(x)|\ge 0\quad\quad \mathrm{on}\quad\quad\{d_{n,k}(x)= R\,\lambda_n^{-1/2}\}\cup\partial\Om$$
and
\begin{equation}
\nonumber
L_n(\phi_n-|U_n|)\ge -L_n\,|U_n|=\Delta \,|U_n|-\lambda_n\,|U_n|+p|U_n|^p\ge (p-1)\,|U_n|^p\ge 0
\end{equation}
in $\Om\backslash \{d_{n,k}(x)\le R\,\lambda_n^{-1/2}\}$. Then (for $R$ large and $n\ge n_0(R)$)
we obtain $|U_n|\le \phi_n$ in the same set, by the minimum principle. Moreover, since by \eqref{limtildelam}
$$|U_n(x)|\le\|U_n\|_{L^{\infty}(\Om)}=|U_n(P^1_n)|\le  C \lambda_n^{\frac{1}{p-1}}$$
for some $C>0$,  we also have,  in $\{d_{n,k}(x)\le R\,\lambda_n^{-1/2}\}$,
$$
|U_n(x)|\le\|U_n(x)\|_{L^{\infty}(\Om)}=|U_n(P^1_n)|\le  C e^{\gamma R}\lambda_n^{\frac{1}{p-1}}\sum_{i=1}^k e^{-\gamma\sqrt{\lambda_n}|x-P_n^i|}.
$$
Then, possibly by choosing a larger $C$, estimate \eqref{stimglob} follows for every $n$.
\end{proof}
We now exploit the previous results to show that  suitable rescalings of the solutions to \eqref{eq:auxiliary_n} converge (locally) to some bounded solution $V$ of
\begin{equation}
\label{eqV}
-\Delta  V+ V=| V|^{p-1} V
\end{equation}
in $\R^N$.
\begin{lemma}
\label{lemlim1}
Let \eqref{eq:mainass_secMorse} hold.
Then $|u_n|$ admits $k\le \bar k$ local maxima $P_n^1,...,P_n^k$ in $\Om$ such that, defining
\begin{equation}
u_{i,n}(x)= \Bigl ( \frac{\mu_n}{\lambda_n}\Bigr )^{\frac{1}{p-1}}u_n \bigr (\frac{x}
{\sqrt {\lambda_n}}+P_n^i\bigl ),\quad\quad x\in \Om_{n,i}:=\sqrt{\lambda_n}\bigr (\Om-P_n^i \bigl ),
\end{equation}
it results, up to a subsequence,
\begin{equation}
u_{i,n}(x)\rightarrow V_i\quad\quad \mathrm{in}\,\,\mathcal{C}^1_{\loc}(\R^n)\quad \mathrm{as}\,\,n\to +\infty,\quad \forall\,\,i=1,2,...,k,
\end{equation}
where $V_i$ is a bounded solution of \eqref{eqV} with $m(V_i)\le \bar{k}$.

\noindent As a consequence, for every $q\ge 1$,
\begin{equation}
\label{convlq}
\Bigl ( \frac{\mu_n}{\lambda_n}\Bigr )^{\frac{q}{p-1}}\lambda_n^{N/2}\int_{\Om} |u_n|^q \,dx\rightarrow
\sum_{i=1}^{k}\int_{\R^n}|V_i|^q\,dx\quad\quad \mathrm{as}\,\, n\to +\infty.
\end{equation}
\end{lemma}
\begin{proof}
By Lemma \ref{localblow} we have $\lambda_n\to +\infty$; then, the first part of the lemma follows by definition \eqref{equn}, by \eqref{limblowseqi} and by Proposition \ref{glob1}; by the same proposition
and by Proposition \ref{glob2} we also have that the local maxima $P^i_n$ satisfies \eqref{limpin} and that the pointwise estimate
\begin{equation}
\label{stimglobvn}
|u_n(x)|\le C\Bigl (\frac{\lambda_n}{\mu_n}\Bigr )^{\frac{1}{p-1}} \sum_{i=1}^ke^{-\gamma\sqrt{\lambda_n}|x-P_n^i|}\,,\quad\quad \forall \,x\in\Om,\quad n\in\N\,.
\end{equation}
holds. Let us fix $R>0$ and set $r_n=R/\sqrt{\lambda_n}$; for large enough $n$, \eqref{limpin} implies
$$B_{r_n}(P^i_n)\subset \Om,\quad\quad B_{r_n}(P^i_n)\cap B_{r_n}(P^j_n)=\emptyset, \quad i\neq j.$$
Then we obtain
\[
\begin{array}{cl}
&\left |\left ( \frac{\mu_n}{\lambda_n}\right)^{\frac{q}{p-1}}\lambda_n^{N/2}\int_{\Om} |u_n|^q \,dx-
\sum_{j=1}^{k}\int_{B_R(0)}|u_{j,n}|^q\,dx\,\,\right |
\smallskip\\&
=\left ( \frac{\mu_n}{\lambda_n}\right )^{\frac{q}{p-1}}\lambda_n^{N/2}\left |\int_{\Om} |u_n|^q \,dx-
\sum_{j=1}^{k}\int_{B_{r_n}(P^j_n)}|u_{n}|^q\,dx\,\,\right |
\smallskip\\
&=\left ( \frac{\mu_n}{\lambda_n}\right )^{\frac{q}{p-1}}\lambda_n^{N/2}
\int_{\Om\backslash \bigcup_{j=1}^k\,B_{r_n}(P^j_n)} |u_n|^q \,dx
\le C^q\lambda_n^{N/2}
\int_{\Om\backslash \bigcup_{j=1}^k\,B_{r_n}(P^j_n)} \left |\sum_{i=1}^ke^{-\gamma\sqrt{\lambda_n}|x-P_n^i|}\right |^q \,dx
\smallskip\\
&\le C^qk^{q-1}\lambda_n^{N/2} \sum_{i=1}^k
\int_{\Om\backslash \bigcup_{j=1}^k\,B_{r_n}(P^j_n)} e^{-q\gamma\sqrt{\lambda_n}|x-P_n^i|} \,dx
\smallskip\\
&\le C^qk^{q-1}\lambda_n^{N/2} \sum_{i=1}^k
\int_{\R^N\backslash \,B_{r_n}(P^i_n)} e^{-q\gamma\sqrt{\lambda_n}|x-P_n^i|} \,dx
\smallskip\\
&\le (Ck)^{q}\sum_{i=1}^k
\int_{\R^N\backslash \,B_{R}(0)} e^{-q\gamma\,|y|} \,dy\le C_1 \,e^{-C_2 R},
\end{array}
\]
for some positive $C_1$, $C_2$. Letting $n\to +\infty$ we have, up to subsequences,
\begin{multline*}
\Bigg |\lim_{n\to +\infty}\Bigl ( \frac{\mu_n}{\lambda_n}\Bigr )^{\frac{q}{p-1}}\lambda_n^{N/2}\int_{\Om} |u_n|^q \,dx-
\sum_{i=1}^{k}\int_{B_R(0)}|V_i|^q\,dx\,\,\Bigg |
\\
=\lim_{n\to +\infty}\Bigg |\Bigl ( \frac{\mu_n}{\lambda_n}\Bigr )^{\frac{q}{p-1}}\lambda_n^{N/2}\int_{\Om} |u_n|^q \,dx-
\sum_{i=1}^{k}\int_{B_R(0)}|u_{i,n}|^q\,dx\,\,\Bigg |\le C_1 \,e^{-C_2 R}.
\end{multline*}
Then, \eqref{convlq} follows by taking $R\to +\infty$.
\end{proof}
The previous lemma allows us to gain some information on the asymptotic behavior of the sequences $\lambda_n$, $\mu_n$ and $\|u_n\|_{L^{p+1}(\Omega)}$. We first provide some bounds for the solutions of the limit problem \eqref{eqV} which will be useful in the sequel.
\begin{lemma}
\label{boundbelow}
Let $V_i$, $i=1,\dots,k$ be as in Lemma \ref{lemlim1} (so that $m(V_i)\leq\bar k$). There exists a constant $C$, only depending on the full sequence $\{u_n\}_n$ and not on $V_i$
(and on the particular associated subsequence), such that
\[
\|V_i\|_{H^1}^2 = \|V_i\|_{L^{p+1}}^{p+1} \leq C.
\]
Furthermore, if also $m(V_i)\geq2$ (or, equivalently, if $V_i$ changes sign)
the following estimates hold:
\begin{equation}
\label{uppstiml2}
\|V_i\|^{p+1}_{L^{p+1}}> 2\,\|Z\|^{{p+1}}_{L^{{p+1}}},\qquad
\|V_i\|^2_{L^2}> 2\,\|Z\|^{2}_{L^{2}},
\end{equation}
where $Z\equiv Z_{N,p}$ is the unique positive solution to \eqref{eqV}.
\end{lemma}
\begin{proof}
To prove the bounds from above we claim that there exists $\bar R>0$, not depending on $i$, such that $V_i$ is stable outside $\overline{B_{\bar R}}$. Then
the desired estimate will follow, since
\[
\|V_i\|^{p+1}_{L^{p+1}} = \int_{B_{\bar R}} |V_i|^{p+1} + \int_{\R^N\setminus B_{\bar R}} |V_i|^{p+1},
\]
where the first term is uniformly bounded by Theorem \ref{thm:unif_est_Farina}, while the second one can be estimated in an uniform way
by reasoning as in the proof of \cite[Theorem 2.3]{MR2825606}. To prove the claim, recalling \eqref{defhn} and \eqref{limhr0}, let
$\bar R$ be such that
\[
h_k(\bar R) \leq \left(\frac{1}{p}\right)^{1/(p-1)}.
\]
Then $|V_i(x)|^{p-1}\leq 1/p $ on $\R^N\setminus B_{\bar R}$ and thus, for any $\psi\in C^\infty_0(\R^N)$, $\psi\equiv0$ in $B_{\bar R}$, it holds
\[
\int_{\R^N} |\nabla \psi|^2 +  \psi^2 - p|V_i|^{p-1}\psi^2\,dx \geq \left( 1 - p \|V_i\|^{p-1}_{L^{\infty}(\R^N\setminus B_{\bar R})}\right)\int_{\R^N} \psi^2 \geq 0.
\]
Hence $V_i$ is stable outside $B_{\bar R}$, and the first part of the lemma follows.

On the other hand, if $V_i$ is a sign-changing solution to \eqref{eqV}, the associated energy functional
\begin{equation}
\nonumber
E(V_i)= \frac{1}{2}\|\nabla V_i\|^2_{L^2}+\frac{1}{2}\|V_i\|^2_{L^2}-\frac{1}{p+1}\|V_i\|^{p+1}_{L^{p+1}}
\end{equation}
 satisfies the following \emph{energy doubling property} (see \cite{MR2263672}):
 $$E(V_i)>2\,E(Z)$$
 On the other hand, by using the equation $E'(V_i)V_i=0$ and the Pohozaev identity  one gets
 \begin{equation}
\label{eulp}
\|V_i\|^{p+1}_{L^{p+1}}= 2\,\frac{p+1}{p-1}\,E(V_i),\qquad
\|V_i\|^2_{L^2}= \frac{N+2-p\,(N-2)}{p-1}\,E(V_i)
\end{equation}
Since the ground state solution $Z$ satisfies the same identities, the bounds \eqref{uppstiml2} are readily verified.
\end{proof}
\begin{proposition}
Let \eqref{eq:mainass_secMorse} hold and the functions $V_i$ be defined as in Lemma
\ref{lemlim1}. We have, as $n\to +\infty$,
\begin{eqnarray}
\label{convl2}
{\mu_n}^{\frac{2}{p-1}}\,\lambda_n^{N/2-2/(p-1)}&\longrightarrow
\sum_{i=1}^{k}\int_{\R^n}|V_i|^2\,dx
\\
\label{convlp}
{\mu_n}^{\frac{p+1}{p-1}}\,\lambda_n^{N/2-(p+1)/(p-1)}\int_{\Om} |u_n|^{p+1} \,dx&\longrightarrow
\sum_{i=1}^{k}\int_{\R^n}|V_i|^{p+1}\,dx
\\
\label{convl2grad}
\alpha_n\,{\mu_n}^{\frac{2}{p-1}}\,\lambda_n^{N/2-(p+1)/(p-1)}&\longrightarrow
\sum_{i=1}^{k}\int_{\R^n}|\nabla V_i|^2\,dx.
\end{eqnarray}
\end{proposition}
\begin{proof}
The limits \eqref{convl2} and \eqref{convlp} follow respectively by choosing $q=2$ and $q=p+1$ in \eqref{convlq} (recall that $\|u_n\|_{L^{2}}=1$). Furthermore, from the equations
for $u_n$ and $V_k$, we have
\[
\alpha_n+\lambda_n=\mu_n\|u_n\|_{L^{p+1}}^{p+1},\qquad
\int_{\R^n}|\nabla V_i|^2\,dx + \int_{\R^n}|V_i|^2\,dx = \int_{\R^n}|V_i|^{p+1}\,dx,
\]
and also \eqref{convl2grad} follows.
\end{proof}
\begin{corollary}
\label{limmass}
With the same assumptions as above, we have that
\begin{enumerate}
  \item if $1<p<1+\frac{4}{N}$, then $\mu_n\to +\infty$
  \item if $p=1+\frac{4}{N}$, then $\mu_n\to \big (\sum_{i=1}^{k}\|V_i\|_{L^2}^2\big )^{2/N}\ge k^{2/N}
  \|Z\|_{L^2}^{4/N}$
  \item if $1+\frac{4}{N}<p<2^*-1$, then $\mu_n\to 0$.
\end{enumerate}
Furthermore
\begin{equation}
\label{limalphalam}
\frac{\alpha_n}{\lambda_n}\longrightarrow \frac{N(p-1)}{N+2-p(N-2)}.
\end{equation}
\end{corollary}
\begin{proof}
The limits of $\mu_n$ follow by the previous proposition. To prove the lower bound in $2$, recall that either $V_i=Z$ or $V_i$ satisfies \eqref{uppstiml2}. Finally, taking the quotient between \eqref{convl2grad} and \eqref{convl2}, we have
$$
\frac{\alpha_n}{\lambda_n}\longrightarrow
\frac{\sum_{i=1}^{k}\int_{\R^n}|\nabla V_i|^2\,dx}{\sum_{i=1}^{k}\int_{\R^n}| V_i|^2\,dx}
$$
On the other hand, for every $i=1,2,...,k$ it holds
$$
\|\nabla V_i\|_{L^2}^2=\Bigg (\frac{\| V_i\|_{L^{p+1}}^{p+1}}{\|V_i\|_{L^2}^2} -1 \Bigg )\|V_i\|_{L^2}^2=
\frac{N(p-1)}{N+2-p(N-2)}\, \|V_i\|_{L^2}^2
$$
where the last equality follows by \eqref{eulp}. By inserting this into the above limit, we get \eqref{limalphalam}.
\end{proof}
\begin{proof}[Proof of Theorem \ref{thm:bbd_index}]
Let $(U_n,\lambda_n)$ solve \eqref{eq:main_prob_U}, with $\rho=\rho_n\to  +\infty$
and $m(U_n)\leq k$. Changing variables as in \eqref{eq:main_prob_u}, we have that
$u_n=\rho_n^{-1/2}U_n$ satisfies \eqref{eq:auxiliary_n} with $\mu_n = \rho_n^{(p-1)/2} \to
+\infty$. As a consequence, Lemma \ref{lemma:case_alpha_n_bounded} guarantees that
$\alpha_n\to+\infty$, and Corollary \ref{limmass} yields $p<1+4/N$.

On the other hand, by direct minimization of the energy one can show that, if $p<1+4/N$, for
every $\rho>0$ there exists a solution of \eqref{eq:main_prob_U} having Morse index one (see
also Section \ref{sec:1const}).
\end{proof}
\begin{remark}
\label{limGN}
Reasoning as above we can also show that
\begin{equation}
\label{newcnp1}
\frac{\int_{\Om} |u_n|^{p+1} \,dx}{\alpha_n^{N(p-1)/4}}\longrightarrow
C_{N,p}\,\frac{\|Z\|_{L^2}^{p-1}}{\big (\sum_{i=1}^{k}\| V_i\|_{L^2}^2\big )^{(p-1)/2}}.
\end{equation}
\end{remark}

\section{Max-min principles with two constraints}\label{sec:2const}

In this section we deal with the maximization problem with two constraints introduced in \cite{MR3318740}, aiming at considering more general max-min classes of critical points.
Let $\Mcal$ be defined in \eqref{Emu} and, for any fixed $\alpha>\lambda_1(\Om)$, let $\mathcal{B}_\alpha$, $\mathcal{U}_\alpha$ be defined as in \eqref{eq:defBU}. We will look for critical points of the $\mathcal{C}^2$ functional
\[
f(u)=\int_{\Omega}|u|^{p+1},\quad\quad\quad u\in \Mcal,
\]
constrained to $\mathcal{U}_\alpha$. To start with, we notice that the topological
properties of such set depend on $\alpha$.
\begin{lemma}\label{lemma:tilde_U_manifold}
Let $\alpha>\lambda_1(\Om)$. Then the set
\[
\Ucal_{\alpha}\setminus\left\{ \varphi\in \Ucal_\alpha :
-\Delta\varphi = \alpha\varphi\right\}
\]
is a smooth submanifold of $H^1_0(\Omega)$ of codimension 2. In particular, this property
holds true for $\Ucal_\alpha$ itself, provided $\alpha\neq\lambda_k(\Om)$, for every $k$.
\end{lemma}
\begin{proof}
Let us set $F(u)=(\int_\Omega u^2\,dx-1, \ \int_\Omega|\nabla u|^2\,dx)$. For every
$u\in\Ucal_\alpha$, if the range of $F'(u)$ is $\R^2$ then $\Ucal_\alpha$ is a smooth manifold
at $u$. Since
\[
F'(u)[v]=2\left(\int_\Omega uv\,dx, \ \int_\Omega\nabla u\cdot\nabla v\,dx\right),
\qquad\text{for every }v\in H^1_0(\Omega),
\]
and $F'(u)[u]=2(1,\alpha)$, we have that $F'(u)$ is not surjective if and only if
\[
\int_\Omega\nabla u\cdot\nabla v\,dx = \alpha \int_\Omega uv\,dx
\qquad\text{for every }v\in H^1_0(\Omega). \qedhere
\]
\end{proof}
\begin{remark}
If $\varphi$ belongs to the eigenspace corresponding to $\lambda_k(\Om)$, then
$\varphi \in \Ucal_{\lambda_k(\Om)}$. As a consequence $\Ucal_{\lambda_k(\Om)}$ may not be smooth
near $\varphi$. For instance, $\Ucal_{\lambda_1(\Om)}$ consists of two isolated points, $\pm\varphi_1$.
\end{remark}
Of course $\Ucal_\alpha$ is closed and odd, for any $\alpha$. Recalling Definition
\ref{def:genus} we deduce that its genus
$\gamma(\Ucal_\alpha)$ is well defined.
\begin{lemma}
If $\alpha<\lambda_{k+1}(\Om)$, for some $k$, then $\gamma(\Ucal_\alpha)\leq k$.
\end{lemma}
\begin{proof}
Let $V_k:=\spann\{\varphi_1,\dots,\varphi_k\}$.
Since
\[
\min\left\{\int_\Omega |\nabla u|^2\,dx : u\in V_k^\perp,\,
\int_\Omega u^2\,dx=1\right\}=\lambda_{k+1}(\Om),
\]
we have that $\Ucal \cap V_k^\perp = \emptyset$, thus the projection
\[
g := \proj_{V_k} \from \Ucal_\alpha \to V_k\setminus\{0\}
\]
is a continuous odd map of $\Ucal_\alpha$ into $V_k\setminus\{0\}$. Now, let $h\from\sphere^{m}\to \Ucal$ be continuous and odd.
Then $g\circ h$ is continuous and odd from $\sphere^{m}$ to $V_k\setminus\{0\}$, and Borsuk-Ulam's Theorem forces $m\leq k-1$.
\end{proof}
\begin{lemma}\label{lemma:genusbigger}
If $\alpha>\lambda_{k}(\Om)$, for some $k$, then $\gamma(\Ucal_\alpha)\geq k$.
\end{lemma}
\begin{proof}
To prove the lemma we will construct a continuous map $h\from \sphere^{k-1} \to \Ucal$. Let
$\ell\in\N$ be such that $\lambda_{\ell+1}(\Om)>\alpha$. For every $i=1,\dots,k$ we define the
functions
\[
u_i:=\left(\frac{\lambda_{\ell+i}(\Om)-\alpha}{\lambda_{\ell+i}(\Om)-\lambda_i(\Om)}\right)^{1/2}\varphi_i
+\left(\frac{\alpha-\lambda_{i}(\Om)}{\lambda_{\ell+i}(\Om)-\lambda_i(\Om)}\right)^{1/2}\varphi_{\ell+i}.
\]
We obtain the following straightforward consequences:
\begin{enumerate}
 \item as $\lambda_i(\Om)<\alpha<\lambda_{\ell+i}(\Om)$, for every $i$, $u_i$ is well defined;
 \item $\int_\Omega u_i^2\,dx=1$, $\int_\Omega |\nabla u_i|^2\,dx=\alpha$;
 \item for every $j\neq i$ it holds $\int_\Omega u_iu_j\,dx=\int_\Omega \nabla u_i\cdot\nabla u_j
 \,dx=0$.
\end{enumerate}
Therefore the map $h\from \sphere^{k-1} \to \Ucal$ defined as
\[
h\from x=(x_1,\dots,x_k) \mapsto \sum_{i=1}^k x_iu_i
\]
has the required properties.
\end{proof}

Now we turn to the properties of the functional $f$. To start with, it satisfies
the Palais-Smale (P.S. for short) condition on $\overline{\mathcal{B}}_{\alpha}$; more precisely, the following
holds.
\begin{lemma}
\label{psball}
Every P.S. sequence $u_n$ for $f\big |_{\overline{\mathcal{B}}_{\alpha}}$ is a P.S.
sequence  for $f\big |_{\mathcal{U}_{\alpha}}$ and has a strongly convergent
subsequence in $\mathcal{U}_{\alpha}$.
\end{lemma}
\begin{proof}
We first show that there are no P.S. sequences in ${\mathcal{B}_{\alpha}}$. In fact, if $u_n$ is such a sequence, there is a sequence of real numbers $k_n$ such that
\begin{equation}
\label{ps}
\int_{\Om}|u_n|^{p-1}u_n\,v-k_n\int_{\Om}u_n\,v=o(1)\,\|v\|_{H^1_0}
\end{equation}
for every $v\in H^1_0(\Om)$. Since $u_n$ is bounded in $H^1_0(\Om)$, there is a subsequence (still denoted by $u_n$)
weakly convergent to $u\in H^1_0(\Om)$; moreover, $u_n$ converges strongly in $L^{p+1}(\Om)$ and in $L^2(\Om)$ to the same limit. By choosing $v=u_n$, we see that $k_n$ is bounded, so that we can also assume that $k_n\rightarrow k$. By taking the limit of \eqref{ps} for $n\to\infty$ we get
\begin{equation}
\nonumber
\int_{\Om}|u|^{p-1}u\,v=k\int_{\Om}u\,v
\end{equation}
for every $v\in H^1_0(\Om)$. Hence $u$ is constant, but this contradicts $u\in \Mcal$.

Now, if $u_n$ is a P.S. sequence for $f$ on ${\mathcal{U}}_{\alpha}$, there are sequences of real numbers $k_n$, $l_n$ such that
\begin{equation}
\label{ps1}
\int_{\Om}|u_n|^{p-1}u_n\,v-k_n\int_{\Om}u_n\,v-l_n\int_{\Om}\nabla u_n\,\nabla v=o(1)\,\|v\|_{H^1_0}.
\end{equation}
It is readily seen that $l_n$ is bounded away from zero, otherwise \eqref{ps1} is equivalent to \eqref{ps} (for some subsequence) and we still reach a contradiction. Then, we can divide both sides by $l_n$ and find that
there are sequences $\{\lambda_n\}_n$, $\{\mu_n\}_n$, with $\mu_n$ bounded, such that
\begin{equation}
\nonumber
\int_{\Om}\nabla u_n\,\nabla v+\lambda_n\int_{\Om}u_n\,v-\mu_n\int_{\Om}|u_n|^{p-1}u_n\,v=o(1)\,\|v\|_{H^1_0}.
\end{equation}
Now, by reasoning as before one finds that
also the sequence $\{\lambda_n\}_n$ is bounded, so that by the relation
$$-\Delta u_n+\lambda_n u_n-\mu_n |u_n|^{p-1}u_n=o(1)\quad \mathrm{in}\,\, H^{-1}(\Omega)$$
and  by the compactness of the embedding $H^1_0(\Omega)\hookrightarrow L^{p+1}(\Omega)$, the P.S. condition holds for the functional $f\big |_{{\mathcal{U}}_{\alpha}}$.
\end{proof}
We can combine the previous lemmas to prove one of the main results stated in the introduction.
\begin{proof}[Proof of Theorem \ref{thm:genus_2constr}]
Lemma \ref{psball} allows to apply standard variational methods (see e.g.
\cite[Thm. II.5.7]{St_2008}). We deduce that
$M_{\alpha,\,k}$ is achieved
at some critical point $u$ of $f\big |_{\mathcal{U}_{\alpha}}$.
This amounts to say that $u$ satisfies \eqref{lagreq}
for some real $\lambda$ and $\mu\neq 0$. We claim that there exists at least one
$u\in f^{-1}(M_{\alpha,\,k})\cap \mathcal{U}_{\alpha}$ such that \eqref{lagreq} holds with $\mu>0$.
Assume by contradiction that for \emph{every} critical point of $f\big |_{\mathcal{U}_{\alpha}}$ at level $M_{\alpha,k}$ it holds $\mu< 0$ in equation \eqref{lagreq}.

Let us define the functional $T:\,H^1_0(\Omega)\to \R$ as
$$T(u)=\frac{1}{2}\int_{\Om}|\nabla u|^2.$$
By denoting with $D$ the Fr\'{e}chet derivative and by $<\,,\,>$ the pairing between $H_0^1$ and its dual $H^{-1}$, our assumption can be restated as follows:

\noindent if there are $u\in f^{-1}(M_{\alpha,\,k})\cap\mathcal{U}_{\alpha}$ and $\mu\neq 0$  such that
\begin{equation}
\label{lagreq1}
\langle DT(u),\phi\rangle=\mu\langle Df(u),\phi\rangle
\end{equation}
for every $\phi\in H^1_0(\Om)$ satisfying $\int_{\Om}\phi u=0$ (that is for every $\phi$ tangent
to $\Mcal$ at $u$) then $\mu<0$.

We stress that both $DT(u)$ and $Df(u)$ in the above equation are bounded away from zero, since there are no Dirichlet eigenfunctions in $\mathcal{U}_{\alpha}$ nor critical points of $f$ on $\Mcal$.
Hence, by denoting with $\nabla_{T\Mcal}$ the gradient of a functional (in $H^1_0$) in the direction tangent to $\Mcal$, if $u\in f^{-1}(M_{\alpha,\,k})\cap \mathcal{U}_{\alpha}$ then $\nabla_{T\Mcal}T(u)$ and
$\nabla_{T\Mcal}f(u)$ \emph{are either  opposite or not parallel}. Moreover, the angle between these (non vanishing) vectors is \emph{bounded away from zero}; otherwise, we would find sequences $u_n\in \mathcal{U}_{\alpha}$, $\mu_n>0$ such that
\begin{equation}
\label{noparal}
(\nabla_{T\Mcal}T(u_n),v)_{H^1_0}-\mu_n(\nabla_{T\Mcal}f(u_n),v)_{H^1_0}=o(1)\|v\|_{H^1_0}
\end{equation}
for every $v\in H^1_0(\Om)$; but since
$$(\nabla_{T\Mcal}T(u_n),v)_{H^1_0}=\int_{\Om}\nabla u_n\,\nabla v-\lambda_n^T\int_{\Om}u_n\, v\,,$$
$$(\nabla_{T\Mcal}f(u_n),v)_{H^1_0}=\int_{\Om}|u_n|^{p-1}u_n\,v-\lambda_n^f\int_{\Om}u_n\, v\,,$$
for suitable bounded sequences $\lambda_n^T$, $\lambda_n^f$, this is equivalent to saying that $u_n$ is a P.S. sequence for
$f\big |_{\mathcal{U}_{\alpha}}$, so that, by Lemma \ref{psball}, we would get a constrained critical point with $\mu>0$.

Then, by choosing suitable linear combinations of the above tangential components
one can define a bounded $\mathcal{C}^1$ map $u\mapsto v(u)\in H_0^1(\Om)$, with $v(u)$ tangent to $\Mcal$
and satisfying the following property: there is $\delta>0$ such that
\begin{equation}
\label{diseqv}
\int_{\Om}\nabla u\,\nabla v(u)< -\delta\, ,\quad\quad \int_{\Om}|u|^{p-1}u\,v(u)>\delta\,,
\end{equation}
for every $u\in f^{-1}(M_{\alpha,\,k})\cap \mathcal{U}_{\alpha}$. By continuity and possibly by decreasing $\delta$, inequalities \eqref{diseqv} extend to
\begin{equation}
\label{diseqv1}
f^{-1}(M_{\alpha,\,k}-\bar\eps, M_{\alpha,\,k}+\bar\eps)\cap \big (\overline{\mathcal{B}}_{\alpha}
\backslash \overline{\mathcal{B}}_{\alpha-\tau}\big )
\end{equation}
for small enough, positive $\bar\eps$ and $\tau$. Finally, since there are no critical points of $f$ in
${\mathcal{B}}_{\alpha}$ we can take that the \emph{second of \eqref{diseqv} holds on}
\begin{equation}
\label{diseqv2}
f^{-1}(M_{\alpha,\,k}- \bar\eps, M_{\alpha,\,k}+\bar\eps)\cap \overline{\mathcal{B}}_{\alpha}.
\end{equation}
 Let $\varphi$ be a $\mathcal{C}^1$ function on $\R$ such that:
$$0\le\varphi\le 1, \quad\varphi\equiv 1\,\, \mathrm{in}\,\,(M_{\alpha,\,k}-\bar\eps/2, M_{\alpha,\,k}+\bar\eps/2),\quad \varphi\equiv 0\,\, \mathrm{in}\,\,\R\backslash (M_{\alpha,\,k}-\bar\eps, M_{\alpha,\,k}+\bar\eps),$$
and define
\begin{equation}
\label{vectfield}
e(u)=\varphi(f(u))\,v(u).
\end{equation}
Clearly,  $e$ is a $\mathcal{C}^1$ vector field on $\Mcal$ and is uniformly bounded, so that there exists a global solution $\Phi(u,t)$ of the initial value problem
$$\partial_t\Phi(u,t)=e\big (\Phi(u,t)),\quad\quad \Phi(u,0)=0.$$
By definition \eqref{vectfield} and by the first of \eqref{diseqv} (on \eqref{diseqv1}) we get
$\Phi(u,t_0)  \in \overline{\mathcal{B}}_{\alpha}$ for $t_0> 0$ and for any $u\in \overline{\mathcal{B}}_{\alpha}$; moreover, by the second inequality of \eqref{diseqv} (on \eqref{diseqv2}) there exists $\eps\in (0,\bar\eps)$ such that
$$f(\Phi(u,t_0))>M_{\alpha,\,k}+\eps$$ for every $u\in f^{-1}(M_{\alpha,\,k}-\eps, +\infty)\cap \overline{\mathcal{B}}_{\alpha}$.

Now, by \eqref{maxmin}, there is $A_{\eps}\subset \overline{\mathcal{B}}_{\alpha}$ such that $\gamma(A_{\eps})\ge k$ and
$$\inf_{u\in A_{\eps}} f(u)\ge M_{\alpha,\,k}-\eps.$$
Hence, $\gamma\big (\Phi(A_{\eps},t_0) \big )\ge k$ and
$$\inf_{u\in \Phi(A_{\eps},t_0)} f(u)\ge M_{\alpha,\,k}+\eps$$
contradicting the definition of $M_{\alpha,\,k}$.
\end{proof}
\begin{remark}
\label{rem1}
If $\mu>0$, by testing \eqref{lagreq} with $u$ and by integration by parts we readily get
$\lambda>-\alpha$. An alternative lower bound, independent of $\alpha$, could be obtained
by adapting arguments from \cite{MR968487,MR954951,MR991264} in order to prove
that the Morse index of $u$ (as a solution of \eqref{lagreq}) is less or equal than $k$.
Then Lemma \ref{lem:lambda_bdd_below} would provide $\lambda\geq-\lambda_{k}$.
\end{remark}
\begin{remark}\label{rem:MvsCNp}
By the Gagliardo-Nirenberg inequality \eqref{sobest} we readily obtain that, for every $k\geq1$,
\[
M_{\alpha,k}\leq C_{N,p} \alpha^{N(p-1)/4}.
\]
Taking into account the previous remark, this agrees with Remark \ref{limGN}.
\end{remark}
We conclude this section with the following estimate.
\begin{lemma}
\label{lem:M3vsM1}
Under the assumptions and notation of Theorem \ref{thm:genus_2constr},
\[
M_{\alpha,3} \leq 2^{-(p-1)/2} M_{\alpha,1}.
\]
\end{lemma}
\begin{proof}
Let $A\in \Sigma^{(3)}_{\alpha}$, according to Definition \ref{def:genus}. Notice
that the map
\[
A\ni u \mapsto \left(\int_\Omega |u|u , \int_\Omega |u|^p u\right)\in\R^2
\]
is continuous and equivariant. By the Borsuk-Ulam Theorem, we deduce the existence of
$u_a\in A$ such that
\[
\int_\Omega |u^+_a|^2 = \int_\Omega |u^-_a|^2 = \frac12,\qquad
\int_\Omega |u^+_a|^{p+1} = \int_\Omega |u^-_a|^{p+1} =
\frac12 \int_\Omega |u_a|^{p+1},
\]
while
\[
\text{either }\int_\Omega |\nabla u^+_a|^2 \leq \frac{\alpha}{2}
\qquad
\text{or }\int_\Omega |\nabla u^-_a|^2 \leq \frac{\alpha}{2}.
\]
For concreteness let us assume that the first alternative holds;
as a consequence, we obtain that $v:=\sqrt2 u_a^+$ belongs to
$\overline{\mathcal{B}}_\alpha$. This yields
\[
M_{\alpha,1}\geq \int_\Omega |v|^{p+1} = 2^{(p+1)/2} \int_\Omega |u_a^+|^{p+1}
= \frac{2^{(p+1)/2}}{2} \int_\Omega |u_a|^{p+1} \geq 2^{(p-1)/2}
\inf_{u\in A} \int_\Omega |u|^{p+1},
\]
and since $A\in \Sigma^{(3)}_{\alpha}$ is arbitrary the proposition follows.
\end{proof}

\section{Min-max principles on the unit sphere in
\texorpdfstring{$L^2$}{L\texttwosuperior}}\label{sec:1const}

According to equation \eqref{Emu}, let $\Mcal\subset H^1_0(\Omega)$ denote the unit
sphere with respect to the $L^2$ norm and $\mathcal{E}_{\mu}$ the energy functional
associated to \eqref{eq:main_prob_u}. In this section we are concerned with critical
points of $\mathcal{E}_{\mu}$ on $\Mcal$ (which, in turn, correspond to solutions of our
starting problem \eqref{eq:main_prob_U}).

By the Gagliardo-Nirenberg inequality \eqref{sobest},
setting $\|\nabla u\|^2_{L^2}=\alpha$, one obtains
\begin{equation}\label{eq:boundonboundEmu}
\frac12\,\alpha- \mu\frac{C_{N,p}}{p+1}\,\alpha^{N(p-1)/4}
\leq \mathcal{E}_{\mu}(u)\le \frac12\alpha.
\end{equation}
In particular, $\mathcal{E}_{\mu}$ is bounded on any bounded subset of $\Mcal$,
and it is bounded from below (and coercive) on the entire $\Mcal$ for {subcritical} $p<1+4/N$
and for {critical} $p=1+4/N$ whenever $\mu< \frac{p+1}{2}C_{N,p}^{-1}$ .
In these cases, one can easily show that $\mathcal{E}_{\mu}$ satifies the P.S. condition
and apply the classical {minimax principle for even functionals} on a closed symmetric
submanifold (see e.g. \cite[Thm. II.5.7]{St_2008}).

In the complementary case, when $p$ is either supercritical, i.e. $p>1+4/N$, or
critical and $\mu$ is large, then $\mathcal{E}_{\mu}$ is not bounded from below (see e.g.
\eqref{minusinfty} below). In order to provide a minimax principle suitable for this case,
we recall the Definition \ref{def:genus} of genus and
that of $\mathcal{B}_\alpha$ (see equation \eqref{eq:defBU}). Furthermore, we denote with $K_{c}$ the (closed and symmetric) set of critical points of
$\Ecal_\mu$ at level $c$ contained in $\mathcal{B}_\alpha$. The following theorem is an adaptation of well known arguments
of previous critical point theorems relying on index theory.
\begin{theorem}
\label{infsupteo}
Let $k\ge1$, $\alpha>\lambda_k(\Omega)$, $\mu>0$ and $\tau>0$ be fixed, and let $c_k$ be defined as
in Theorem \ref{thm:genus_1constr}, equation \eqref{infsuplev}. If
\begin{equation}
\label{ass2}
 c_k < \hat c_k:= \inf_{\substack{A\in\Sigma^{(k)}_{\alpha}\\
A\setminus\mathcal{B}_{\alpha-\tau}\neq\emptyset }}
\sup_{A\setminus\mathcal{B}_{\alpha-\tau}}\Ecal_\mu,
\end{equation}
then $K_{c_k}\neq\emptyset$, and it contains a critical point of Morse index less or equal to $k$.
\end{theorem}
\begin{remark}
In case assumption \eqref{ass2} holds for $k,k+1,\dots,k+r$, and $c=c_k=...c_{k+r}$,
then it is standard to extend Theorem \ref{infsupteo} to obtain
\begin{equation}
\label{indexK}
\gamma(K_c)\ge r+1,
\end{equation}
so that $K_c$ contains infinitely many critical points.
\end{remark}
\begin{proof}[Proof of Theorem \ref{infsupteo}]
For any $a\in\R$ we denote by $\Mcal_a$ the sublevel set $\{\mathcal{E}_\mu<a\}$.
First of all we notice that both $c_k$ and $\hat c_k$ are well defined and finite,
by Lemma \ref{lemma:genusbigger} and equation \eqref{eq:boundonboundEmu}.
Suppose now by contradiction that $K_{c_k}=\emptyset$. By a suitably modified version of
the Deformation Lemma (recall that $\Ecal_\mu$ satisfies the P.S. condition on $\Mcal$),
there exist $\delta>0$ and an equivariant homeomorphism
$\eta$ such that
$\eta(u)=u$ outside
$\mathcal{B}_{\alpha}\cap \Mcal_{c_k+2\delta}$ and
\begin{equation}
\label{lowlev}
\eta({\Mcal_{c_k+\delta}\cap \mathcal{B}_{\alpha-\tau}})\subset \Mcal_{c_k-\delta}\cap
\mathcal{B}_{\alpha}.
\end{equation}
By definition of $c_{k}$ there exists $A\in \Sigma^{(k)}_{\alpha}$ such that
$A\subset \Mcal_{c_k+\delta}$; it follows by assumption \eqref{ass2} (and by decreasing $
\delta$ if necessary) that $A\subset \Mcal_{c_k+\delta}\cap \mathcal{B}_{\alpha-\tau}$.
Then, since $\eta$ is an odd homeomorphism, $\eta(A)\in \Sigma^{(k)}_{\alpha}$
and, by definition, $\sup_{\eta(A)}\Ecal_\mu \ge c_k$, in contradiction with
\eqref{lowlev}. Finally, the estimate of the Morse index is a direct consequence of the
definition of genus we deal with: see \cite{MR968487}, Proposition on page 1030, or the discussion
at the end of Section 2 in \cite{MR991264}.
\end{proof}
We now provide a sufficient condition to guarantee the validity of assumption
\eqref{ass2}.
\begin{lemma}\label{lem:ckMak}
Let $k\ge1$, $\alpha>\lambda_k(\Omega)$ and $\mu>0$ satisfy
\begin{equation}
\label{muboundef}
0<\mu<\frac{p+1}{2}\,\frac{\alpha-\lambda_{k}(\Omega)}{M_{\alpha,k}-|\om|^{-\frac{p-1}{2}}}
\end{equation}
where $M_{\alpha,k}$ is defined in Theorem \ref{thm:genus_2constr}.
Then, for $\tau>0$ sufficiently small, \eqref{ass2} holds true.
\end{lemma}
\begin{proof}
We first estimate $c_k$ from above. To this aim, we construct a subset
$\tilde A\in \Sigma^{(k)}_{\alpha-\tau}$ (for any $\tau$ sufficiently small)
as
\begin{equation}
\label{Atilde}
\tilde A= \left\{\sum_{i=1}^k x_i\varphi_i : x=(x_1,\dots,x_k)\in\sphere^{k-1}\right\},
\end{equation}
where, as usual $\varphi_i$ denotes the Dirichlet eigenfunction associated to
$\lambda_i(\Omega)$. Indeed $\gamma(\tilde A)=k$ (it is homeomorphic to
$\sphere^{k-1}$), and $\max_{u\in\tilde A}\|u\|^2_{H^1_0}=\lambda_{k}(\Omega)<\alpha-\tau$ for
$\tau$ small. Hence Holder inequality yields
\begin{equation}
\label{boundabove1}
c_k \leq \sup_{\tilde A}\mathcal{E}_{\mu}\le \frac{1}{2}\lambda_{k}(\Omega)
- \frac{\mu}{p+1}\,|\om|^{-\frac{p-1}{2}}.
\end{equation}
On the other hand, let $A\in\Sigma^{(k)}_{\alpha}$. Theorem \ref{thm:genus_2constr}
implies
\[
\inf_{u\in A} \int_\Omega |u|^{p+1} \leq M_{\alpha,k}.
\]
If moreover $A\setminus\mathcal{B}_{\alpha-\tau}\neq\emptyset$ we infer
\[
\sup_{A\setminus\mathcal{B}_{\alpha-\tau}}\mathcal{E}_{\mu}\ge
\frac{1}{2}(\alpha - \tau)  - \frac{\mu}{p+1} M_{\alpha,k},
\]
and taking the infimum an analogous inequality holds true for $\hat c_k$. Comparing
with \eqref{boundabove1} the lemma follows.
\end{proof}
Exploiting the results above, we are ready to prove our main existence results.
\begin{proof}[End of the proof of Theorem \ref{thm:genus_1constr}]
By Theorem \ref{infsupteo} and Lemma \ref{lem:ckMak} the proof is completed by choosing
\[
\hat\mu_k:=\sup_{\alpha>\lambda_k(\Omega)} \frac{p+1}{2}\,\frac{\alpha-\lambda_{k}(\Omega)}{M_{\alpha,k}-|\om|^{-\frac{p-1}{2}}}.
\qedhere
\]
\end{proof}
\begin{proof}[Proof of Theorem \ref{thm:intro_GS}]
We write the proof in terms of $\Ecal_\mu$, the theorem following by the relations in
\eqref{eq:main_prob_u}. Recall that, for every $u\in \overline{\mathcal{B}}_\alpha$, $\gamma
\left(\{u,-u\}\right)=1$. We deduce that $c_1$ is actually a local
minimum for $\Ecal_\mu$, achieved by some $u$ which solves \eqref{eq:main_prob_u}
(for a suitable $\lambda$), and it can be chosen positive by symmetry.
Since
\[
\int_\Omega |\nabla u|^2 + \lambda u^2 - p\mu|u|^{p+1}\,dx=-(p-1)\int_\Omega \mu|u|^{p+1}\,dx<0,
\]
and $H^1_0(\Omega) = \spann\{u\}\oplus T_{u}\mathcal{M}$, we have that $u$
has Morse index $1$. In a standard way, the minimality
property of $u$ implies also orbital stability  of the associated solitary wave (see e.g.
\cite{MR677997}). Turning to the estimates for $\hat\mu_1 = \hat\rho_1^{(p-1)/2}$,
we can deduce it using Lemma \ref{lem:ckMak} and Remark \ref{rem:MvsCNp}, which yield
\[
\hat\mu_1\left(\Omega,p\right):=\sup_{\alpha>\lambda_1(\Omega)} \frac{p+1}{2}\,\frac{\alpha-\lambda_{1}(\Omega)}{C_{N,p} \alpha^{\frac{N(p-1)}{4}}-|\om|^{-\frac{p-1}{2}}}
\geq \frac{p+1}{2C_{N,p}}\,\sup_{\alpha>\lambda_1(\Omega)}\frac{\alpha-\lambda_{1}
(\Omega)}{\alpha^{\beta}},
\]
where $\beta:=N(p-1)/4$. Now, if $\beta\leq1$ we obtain the desired bound for the
subcritical and critical cases. On the other hand, when $\beta>1$, elementary
calculations show that
\[
\hat\mu_1\left(\Omega,p\right)\geq \frac{p+1}{2C_{N,p}}
\,\frac{(\beta-1)^{(\beta-1)}}{\beta^\beta}\, \lambda_1(\Omega)^{-(\beta-1)},
\]
and finally
\[
\hat\rho_1\left(\Omega,p\right)\geq \underbrace{\left[\frac{p+1}{2C_{N,p}}
\,\frac{(\beta-1)^{(\beta-1)}}{\beta^\beta}\right]^{\frac{2}{p-1}}}_{D_{N,p}}\,
\lambda_1(\Omega)^{\frac{2}{p-1}-\frac{N}{2}}.\qedhere
\]
\end{proof}
\begin{proof}[Proof of Proposition \ref{thm:intro_3>1}]
As usual, by \eqref{eq:main_prob_u}, we have to prove that
\[
\hat\mu_3\left(\Omega,p\right)\geq
2^{(p-1)/2} D_{N,p}\lambda_3(\Omega)^{\frac{2}{p-1}-\frac{N}{2}}.
\]
By Lemmas \ref{lem:ckMak}, \ref{lem:M3vsM1}, and Remark  \ref{rem:MvsCNp} we obtain
\[
\begin{split}
\hat\mu_3 &= \sup_{\alpha>\lambda_3(\Omega)} \frac{p+1}{2}\,\frac{\alpha-\lambda_{3}(\Omega)}{M_{\alpha,3}-|\om|^{-\frac{p-1}{2}}} \geq
\sup_{\alpha>\lambda_3(\Omega)} \frac{p+1}{2}\,\frac{\alpha-\lambda_{3}(\Omega)}{
2^{-(p-1)/2}M_{\alpha,1}-|\om|^{-\frac{p-1}{2}}}\\
&\geq 2^{(p-1)/2}\sup_{\alpha>\lambda_3(\Omega)} \frac{p+1}{2}\,\frac{\alpha-\lambda_{3}
(\Omega)}{C_{N,p}\alpha^\beta-2^{(p-1)/2}|\om|^{-\frac{p-1}{2}}},
\end{split}
\]
where $\beta:=N(p-1)/4$, and the desired result follows by arguing as in the proof of
Theorem \ref{thm:intro_GS}.
\end{proof}
To conclude this section we prove that in the supercritical case, if $\mu$ is not too
large, in addition to $(c_k)_k$ there is a further sequence of critical levels
$(\bar c_k)$ of $\mathcal{E}_{\mu}$ constrained to $\mathcal{M}$. For concreteness,
let us first consider the case $k=1$: since in such case $c_1$ is a local minimum of
$\Ecal_\mu$ in $\mathcal{M}$, and $\Ecal_\mu$ is unbounded from below in $\mathcal{M}$,
the critical level $\bar c_1$ is of mountain pass type.
\begin{proposition}
\label{mpcritlev}
Let $p>1+4/N$, $\mu<\hat\mu_1$, and $u_1$ denote the local minimum point of $\Ecal_\mu$
in $\mathcal{M}$, according to Theorems \ref{infsupteo} and \ref{thm:intro_GS}. The value
\[
\bar c_1 : =\inf_{\gamma\in \Gamma}\sup_{[0,1]}\mathcal{E}_{\mu}(\gamma(s)),
\quad\text{where }\Gamma:=\left\{\gamma\in C([0,1];\Mcal) : \gamma(0)=u_1,\,\gamma(1)<c_1-1\right\},
\]
is a critical level for $\Ecal_\mu$ in $\mathcal{M}$.
\end{proposition}
\begin{proof}
Notice that, if $p>1+4/N$, then  $\mathcal{E}_{\mu}\to -\infty$ along some sequence in
$\Mcal$. Indeed, by defining
\begin{equation}
\label{concfun}
w_n(x): = \eta(x)Z_{N,p}\big ((x-x_0)/a_n\big )\quad \text{and}\quad
\tilde{w_n}:=\frac{w_n}{\| w_n\|^2_{L^2(\Omega)}}\,\in\, \Mcal,
\end{equation}
where $a_n\to 0^+$, $x_0\in \Omega$ and $\eta\in \mathcal{C}_0^{\infty}(\Omega)$, $\eta(x_0)=1$, we obtain
\begin{equation}\label{minusinfty}
\alpha_n :=\|\nabla \tilde w_n\|^2_{L^2(\Omega)}\to +\infty,
\qquad
\frac{\int_{\Omega} |\tilde w_n|^{p+1} \,dx}{\alpha_n^{N(p-1)/4}}\to C_{N,p},
\qquad
\mathcal{E}_{\mu}(\tilde w_n)\to -\infty
\end{equation}
for $n\to +\infty$. Since $u_1$ is a local minimum,
the functional $\mathcal{E}_{\mu}$ has a mountain pass structure on $\Mcal$;
by recalling that $\mathcal{E}_{\mu}$ satisfies the P.S. condition the
proposition follows.
\end{proof}
\begin{remark}\label{rem:further_crit_lev}
One can generalize Proposition \ref{mpcritlev} by constructing critical points via a saddle-point theorem in the following way: let us pick $k$ points $x_1, x_2,...,x_k$ in $\Omega$ and consider the corresponding function $\tilde w_i$; we may assume that
$\mathrm{supp}\,\tilde w_i\cap \mathrm{supp}\,\tilde w_j=\emptyset$ for $i\neq j$, so that these functions are orthogonal. Let us now define the subspace
$V_k=\mathrm{span}\{\varphi_1,\dots,\varphi_k;\tilde w_1,...,\tilde w_k\}$; note that dim $V_k=2k$. Let $R$ be an operator (in $L^2(\om)$) such that $R=I$ on $V_k^{\perp}$, $Ru_i=\tilde w_i$, $i=1,2,..,k$. Possibly after permutations, we can choose $R$ such that $R\big |_{V_k}\in SO(2k)$ (actually, there are infinitely many different choices of $R$). Now, since $SO(2k)$ is (arcwise) connected, there is a continuous path
$\tilde{\gamma}:\,[0,1]\rightarrow SO(2k)$ such that $\gamma(0)=I$, $\gamma(1)=R\big |_{V_k}$.
Then, we can define the following map
\[
\gamma:\, [0,1] \times S^{k-1}\rightarrow \Mcal,
\quad\quad \gamma(s;t_1,....,t_k)=\sum_{i=1}^{k}t_i\tilde{\gamma}(s)u_i,\quad
\]
where $\sum_{i=1}^{k}t_i^2=1$. It is clear that $\gamma$ is continuous; moreover,
\[
\gamma(0;t_1,....,t_k)\in \mathrm{span }\{\varphi_1,\dots,\varphi_k\}\cap \Mcal
\quad \mathrm{and}\quad
\gamma(1;t_1,....,t_k)\in \mathrm{span }\{\tilde w_1,\dots,\tilde w_k\}\cap \Mcal.
\]
Then, by denoting with $\Gamma_k$ the set of the above paths, if $\mu$ is sufficiently
small we obtain the critical levels
\[
\bar c_k : =\inf_{\gamma\in \Gamma_k}\sup_{[0,1]\times S^{k-1}}\mathcal{E}_{\mu}(\gamma(s;t_1,....,t_k)).
\]
\end{remark}

\section{Results in symmetric domains}\label{sec:symm}

This section is devoted to the proof of Theorem \ref{pro:symm}, therefore we assume $1+4/N \leq p < 2^*-1$.
We perform the proof in the case of $\Omega=B$, but it will be clear that the main assumption on $\Omega$ is the following:
\begin{itemize}
 \item[\textbf{(T)}] there is a tiling of $\Omega$, made by $h$ copies of a subdomain $D$, in such a way that from
 any solution $U_D$ of \eqref{eq:main_prob_U} on $D$ one can construct, using reflections, a solution $U_\Omega$ of  \eqref{eq:main_prob_U} on $\Omega$.
\end{itemize}
Then $U_\Omega$ has $h$ times the mass of $U_D$, and recalling Theorem \ref{thm:intro_GS} we deduce that \eqref{eq:main_prob_U} on $\Omega$ is solvable for any
$\rho< h \cdot D_{N,p} \lambda_1(D)^{\frac{2}{p-1}-\frac{N}{2}}$. At this point, for a sequence $(D_k,h_k)_k$ of tilings satisfying  \textbf{(T)}, we
obtain the solvability of \eqref{eq:main_prob_U} on $\Omega$ whenever
\[
\rho< h_k \cdot D_{N,p} \lambda_1(D_k)^{\frac{2}{p-1}-\frac{N}{2}},
\]
and if we can show that
\begin{equation}\label{eq:finaltarget}
\frac{ h_k }{ \lambda_1(D_k)^{\frac{N}{2}-\frac{2}{p-1}}} \to +\infty\qquad\text{as }k\to+\infty,
\end{equation}
we deduce the solvability of \eqref{eq:main_prob_U} on $\Omega$ for every mass. Having this scheme in mind, it is easy to prove analogous results
on rectangles and also in other kind of domains.

Then let $B\subset\R^N$ be the ball (w.l.o.g. of radius one), and let
\[
D_k:=\left\{(r\cos\theta, r\sin\theta,x_3,\dots,x_N)\in B: -  \frac{\pi}{k} < \theta < \frac{\pi}{k}\right\}
\]
Then $D_k$ satisfies \textbf{(T)}, with $h_k=k$. In order to estimate $\lambda_1(D_k)$ we observe that, by elementary trigonometry,
\[
B'_k = B_{\frac{\sin(\pi/k)}{\sin(\pi/k)+1}}\left(\frac{1}{\sin(\pi/k)+1},0,0,\dots,0\right) \subset D_k,
\]
and therefore
\[
\lambda_1(D_k) \le \lambda_1(B'_k) \le C k^2,
\]
for some dimensional constant $C=C(N)$ and $k$ large. Then
\[
\frac{ h_k }{ \lambda_1(D_k)^{\frac{N}{2}-\frac{2}{p-1}}}\ge  C \frac{k}{k^{{N}-\frac{4}{p-1}}} = C k^{1-{N}+\frac{4}{p-1}} = C k^{\frac{N-1}{p-1}\left[1+\frac{4}{N-1} - p\right]},
\]
and finally \eqref{eq:finaltarget} holds true whenever $p< 1+\frac{4}{N-1}$, thus completing the proof of Theorem \ref{pro:symm}.

\small

\subsection*{Acknowledgments}
We would like to thank Jacopo Bellazzini, who pointed out that the results in
\cite{MR3318740}, in the supercritical case, can be read in terms of a local minimization. We
would also like to thank Benedetta Noris, who read a preliminary version of this manuscript.
This work is partially supported  by the PRIN-2012-74FYK7 Grant:
``Variational and perturbative aspects of nonlinear differential problems'',
by the ERC Advanced Grant  2013 n. 339958:
``Complex Patterns for Strongly Interacting Dynamical Systems - COMPAT'',
and by the INDAM-GNAMPA group.


\begin{flushright}
{\tt dario.pierotti@polimi.it}\\
{\tt gianmaria.verzini@polimi.it}\\
Dipartimento di Matematica, Politecnico di Milano\\
piazza Leonardo da Vinci 32, 20133 Milano, Italy.
\end{flushright}
\end{document}